\documentclass[11pt]{article}
\usepackage{lmodern}
\usepackage[T1]{fontenc}
\usepackage{hyperref}
\usepackage{amsfonts,amsmath,amssymb,amsopn,amsthm}
\usepackage[all]{xy}
\usepackage{vmargin}
\usepackage{comment}
\title{Constructible isocrystals}
\author{Bernard Le Stum\thanks{bernard.le-stum@univ-rennes1.fr}}

\date{Version of \today}
\newtheorem{thm}{Theorem}[section]
\newtheorem{prop}[thm] {Proposition}
\newtheorem{cor}[thm] {Corollary}
\newtheorem{lem}[thm] {Lemma}
\newtheorem{dfn}[thm] {Definition}
\begin{document}

\maketitle
 
\begin{abstract}
We introduce a new category of coefficients for $p$-adic cohomology called constructible isocrystals.
Conjecturally, the category of constructible isocrystals endowed with a Frobenius structure is equivalent to the category of perverse holonomic arithmetic $\mathcal D$-modules.
We prove here that a constructible isocrystal is completely determined by any of its geometric realizations.
\end{abstract}

\tableofcontents

\addcontentsline{toc}{section}{Introduction}
\section*{Introduction}

The relation between topological invariants and differential invariants of a manifold is always fascinating.
We may first recall de Rham theorem that implies the existence of an isomorphism
$$
\mathrm H^i_{\mathrm{dR}}(X) \simeq \mathrm{Hom}(\mathrm H_{i}(X), \mathbb C)
$$
on any complex analytic manifold.
The non abelian version is an equivalence of categories
$$
\mathrm{MIC}(X) \simeq \mathrm{Rep}_{\mathbb C}(\pi_{1}(X,x))
$$
between coherent modules endowed with an integrable connection and finite dimensional representations of the fundamental group.
The same result holds on a smooth algebraic variety if we stick to regular connections (see \cite{Deligne70} or Bernard Malgrange's lecture in \cite{Borel87}).
It has been generalized by Masaki Kashiwara (\cite{Kashiwara84}) to an equivalence
$$
\mathrm D^{\mathrm b}_{\mathrm{reg,hol}}(X) \simeq \mathrm D^{\mathrm b}_{\mathrm{cons}}(X^{\mathrm {an}})
$$
between the categories of bounded complexes of $\mathcal D_{X}$-modules with regular holonomic cohomology and bounded complexes of $\mathbb C_{X^{\mathrm{an}}}$-modules with constructible cohomology.

Both categories come with a so-called $t$-structure but these $t$-structures do not correspond under this equivalence.
Actually, they define a new $t$-structure on the other side that may be called \emph{perverse}.
The notion of perverse sheaf on $X^\mathrm{an}$ has been studied for some time now (see \cite{Borel87} for example).
On the $\mathcal D$-module side however, this notion only appeared in a recent article of Kashiwara (\cite{Kashiwara04}) even if he does not give it a name (we call it perverse but it might as well be called constructible (see \cite{Abe13*}).
Anyway, he shows that the perverse $t$-structure on $\mathrm  D^b_{\mathrm{reg,hol}}(X)$ is given by
$$
\left\{
\begin{array}{l}
\mathrm D^{\leq 0} : \mathrm{codim}\,\mathrm{supp}\, \mathcal H^n(\mathcal F^\bullet) \geq n\ \mathrm{for}\ n \geq 0 \\
\mathrm D^{\geq 0} :  \mathcal H^n_{Z}(\mathcal F^\bullet) = 0\ \mathrm{for}\  n < \mathrm{codim} Z.
\end{array}
\right.
$$
In particular, if we call \emph{perverse} a complex of $\mathcal D_{X}$-modules satisfying both conditions, there exists an equivalence of categories
$$
\mathrm D^{\mathrm{perv}}_{\mathrm{reg,hol}}(X) \simeq \mathrm{Cons}(X^{\mathrm {an}})
$$
between the categories of perverse (complexes of) $\mathcal D_{X}$-modules with regular holonomic cohomology and constructible $\mathbb C_{X^{\mathrm{an}}}$-modules.

In a handwritten note \cite{Deligne*} called ``Cristaux discontinus'', Pierre Deligne gave an algebraic interpretation of the right hand side of this equivalence.
More precisely, he introduces the notion of constructible pro-coherent crystal and proves an equivalence
$$
\mathrm{Cons}_{\mathrm{reg,pro-coh}}(X/\mathbb C) \simeq \mathrm{Cons}(X^{\mathrm {an}})
$$
between the categories of regular constructible pro-coherent crystals and constructible $\mathbb C_{X^{\mathrm{an}}}$-modules.

By composition, we obtain what may be called the \emph{Deligne-Kashiwara correspondence}
$$
\mathrm{Cons}_{\mathrm{reg,pro-coh}}(X/\mathbb C) \simeq \mathrm D^{\mathrm{perv}}_{\mathrm{reg,hol}}(X).
$$
It would be quite interesting to give an algebraic construction of this equivalence but this is not our purpose here.
Actually, we would like to describe an arithmetic analog.

Let $K$ be a $p$-adic field with discrete valuation ring $\mathcal V$ and perfect residue field $k$.
Let $X \hookrightarrow P$ be a locally closed embedding of an algebraic $k$-variety into a formal $\mathcal V$-scheme.
Assume for the moment that $P$ is smooth and quasi-compact, and that the locus of $X$ at infinity inside $P$ has the form $D \cap \overline X$ where $D$ is a divisor in $P$.
We may consider the category $\mathrm  D^b(X \subset P/K)$ of bounded complexes of $\mathcal D^\dagger_{P}({}^\dagger D)_{\mathbb Q}$-modules on $P$ with support on $\overline X$ (see \cite{Berthelot02} for example).
On the other hand, we may also consider the category of overconvergent isocrystals on $(X \subset P/K)$.
In \cite{Caro09}, Daniel Caro proved that there exists a fully faithful functor
$$
\mathrm{sp}_{+} : \mathrm{Isoc}^\dagger_{\mathrm{coh}}(X \subset P/K) \to \mathrm  D^{\mathrm b}_{\mathrm{coh}}(X \subset P/K)
$$
(the index $\mathrm{coh}$ simply means overconvergent isocrystals in Berthelot's sense - see below).
This is the first step towards an overconvergent Deligne-Kashiwara correspondence.
Note that this construction is extended to a slightly more general situation by Tomoyuki Abe and Caro in \cite{AbeCaro13*} and was already known to Pierre Berthelot in the case $\overline X = P_{k}$ (proposition 4.4.3 of \cite{Berthelot96}).

In \cite{LeStum14}, we defined a category that we may denote $\mathrm{MIC}^\dagger_{\mathrm{cons}}(P/K)$ of convergent constructible $\nabla$-modules on $P_{K}$ when $P$ is a geometrically connected smooth proper curve over $\mathcal V$, as well as a category $\mathrm  D^{\mathrm {perv}}(P/K)$ of perverse (complexes of) $\mathcal D^\dagger_{P\mathbb Q}$-modules on $P$, and built a functor
$$
\mathrm R \widetilde{\mathrm{sp}}_{*} : \mathrm{MIC}^\dagger_{\mathrm{cons}}(P/K) \to \mathrm  D^{\mathrm{perv}}_{\mathrm{coh}}(P/K).
$$
Actually, we proved the overconvergent Deligne-Kashiwara correspondence in this situation:
this functor induces an equivalence of categories
$$
\mathrm R \widetilde{\mathrm{sp}}_{*}  : F\mathrm{-MIC}^\dagger_{\mathrm{cons}}(P/K) \simeq F\mathrm{-D}^{\mathrm{perv}}_{\mathrm{hol}}(P/K)
$$
between (convergent) constructible $F$-$\nabla$-modules on $P_{K}$ and perverse holonomic $F$-$\mathcal D^\dagger_{P\mathbb Q}$-modules on $P$.
Note that this is compatible with Caro's $\mathrm{sp}_{+}$ functor.

In order to extend this theorem to higher dimension, it is necessary to develop a general theory of \emph{constructible (overconvergent) isocrystals}.
One could try to mimic Berthelot's original definition and let $\mathrm{Isoc}^\dagger_{\mathrm{cons}}(X \subset Y \subset P/K)$ be the category of $j^\dagger_{X}\mathcal O_{]Y[}$-modules $\mathcal F$ endowed with an overconvergent connection which are only ``constructible'' and not necessarily coherent (here $X$ is open in $Y$ and $Y$ is closed in $P$).
It means that there exists a locally finite covering of $X$ by locally closed subvarieties $Z$ such that $j^\dagger_{Z}\mathcal F$ is a coherent $j^\dagger_{Z}\mathcal O_{]Y[}$-module.
It would then be necessary to show that the definition is essentially independent of $P$ as long as $P$ is smooth and $Y$ proper, and that they glue when there does not exist any global geometric realization.

We choose here an equivalent but different approach with built-in functoriality.
I introduced in \cite{LeStum11} the overconvergent site of the algebraic variety $X$ and showed that we can identify the category of locally finitely presented modules on this site with the category of overconvergent isocrystals in the sense of Berthelot.
Actually, we can define a broader category of overconvergent isocrystals (without any finiteness condition) and call an overconvergent isocrystal $E$ \emph{constructible} when there exists a locally finite covering of $X$ by locally closed subvarieties $Y$ such that $E_{|Y}$ is locally finitely presented.
Note that $K$ may be any non trivial complete ultrametric field and that there exists a relative theory (over some base $O$).
We denote by $\mathrm{Isoc}^\dagger_{\mathrm{cons}}(X/O)$ the category of constructible overconvergent isocrystals on $X/O$.
 
In order to compute these objects, one may define a category $\mathrm{MIC}^\dagger_{\mathrm{cons}}(X,V/O)$ of constructible modules endowed with an overconvergent connection on any ``geometric realization'' $V$ of $X/O$, as in Berthelot's approach.
We will prove (theorem \ref{thm} below) that, when $\mathrm{Char}(K) \neq 0$, there exists an equivalence of categories
$$
\mathrm{Isoc}^\dagger_{\mathrm{cons}}(X/O) \simeq \mathrm{MIC}^\dagger_{\mathrm{cons}}(X,V/O).
$$
As a corollary, we obtain that the later category is essentially independent of the choice of the geometric realization (and that they glue when there does not exist such a geometric realization).
Note that this applies in particular to the case of the curve $P$ above which ``is'' a geometric realization of $P_{k}$ so that
$$
 \mathrm{Isoc}^\dagger_{\mathrm{cons}}(P_{k}/K) =  \mathrm{MIC}^\dagger_{\mathrm{cons}}(P/K).
$$

In the first section, we briefly present the overconvergent site and review some material that will be needed afterwards.
In the second one, we study some functors between overconvergent sites that are associated to locally closed embeddings.
We do a little more that what is necessary for the study of constructible isocrystal, hoping that this will be useful in the future.
In section three, we introduce overconvergent isocrystals and explain how one can construct and deconstruct them.
In the last section, we show that constructible isocrystals may be interpreted in terms of modules with integrable connections.

\section*{Notations and conventions}

Throughout this article, $K$ denotes a non trivial complete ultrametric field with valuation ring $\mathcal V$ and residue field $k$.

An \emph{algebraic variety} over $k$ is a scheme over $k$ that admits a locally finite covering by schemes of finite type over $k$.
A \emph{formal scheme} over $\mathcal V$ always admits a locally finite covering by $\pi$-adic formal schemes of finite presentation over $\mathcal V$.
An \emph{analytic variety} over $K$ is a strictly analytic $K$-space in the sense of Berkovich (see \cite{Berkovich93} for example).
We will use the letters $X, Y, Z, U, C, D, \ldots$ to denote algebraic varieties over $k$, $P, Q, S$ for formal schemes over $\mathcal V$ and $V, W,O$ for analytic varieties over $K$.
 
An analytic variety over $K$ is said to be \emph{good} if it is locally affinoid.
This is the case for example if $V$ is affinoid, proper or algebraic, or more generally, if $V$ is an open subset of such a variety.
 Note that, in Berkovich original definition \cite{Berkovich90}, all analytic varieties were good.

As usual, we will write $\mathbb A^1$ and $\mathbb P^1$ for the affine and projective lines.
We will also use $\mathbb D(0, 1^\pm)$ for the open or closed disc of radius $1$.

\section*{Acknowledgments} Many thanks to Tomoyuki Abe, Pierre Berthelot, Florian Ivorra, Vincent Mineo-Kleiner, Laurent Moret-Bailly, Matthieu Romagny and Atsushi Shiho with whom I had interesting conversations related to some questions discussed here.

\section{The overconvergent site} \label{bible}
We briefly recall the definition of the overconvergent site from \cite{LeStum11}.
An object is made of
\begin{enumerate}
\item a locally closed embedding $X \hookrightarrow P$ of an algebraic variety (over $k$) into a formal scheme (over $\mathcal V$) and
\item a morphism $\lambda : V \to P_{K}$ of analytic varieties (over $K$).
\end{enumerate}

We denote this object by $X \subset P \stackrel {\mathrm{sp}}\leftarrow P_{K} \leftarrow V$ and call it an \emph{overconvergent variety}.
Here, $\mathrm{sp}$ denotes the \emph{specialization} map and we also introduce the notion of \emph{tube} of $X$ in $V$ which is
$$
]X[_{V} := \lambda^{-1}(\mathrm{sp}^{-1}(X)).
$$
We call the overconvergent variety \emph{good} if any point of $]X[_{V}$ has an affinoid neighborhood in $V$.
It makes it simpler to assume from the beginning that all overconvergent varieties are good since the important theorems can only hold for those (and bad overconvergent varieties play no role in the theory).
But, on the other hand, most constructions can be carried out without this assumption.

We define a \emph{formal morphism} between overconvergent varieties as a triple of compatible morphisms:
$$
\xymatrix{X' \ar@{^{(}->}[r] \ar[d]^f & P' \ar[d]^v & P'_{K} \ar[l] \ar[d]^{v_{K}} & V' \ar[l] \ar[d]^u \\ X \ar@{^{(}->}[r] & P & P_{K} \ar[l] & V \ar[l]
}
$$
Such a formal morphism induces a continuous map
$$
]f[_{u} : ]X'[_{V'} \to ]X[_{V}
$$ 
between the tubes.

Actually, the notion of formal morphism is too rigid to reflect the true nature of the algebraic variety $X$ and it is necessary to make invertible what we call a \emph{strict neighborhood} and that we define now:
it is a formal morphism as above such that $f$ is an isomorphism $X' \simeq X$ and $u$ is an open immersion that induces an isomorphism between the tubes $]X'[_{V'} \simeq ]X[_{V}$.
Formal morphisms admit calculus of right fraction with respect to strict neighborhoods and the quotient category is the \emph{overconvergent site} $\mathrm{An}^\dagger_{/\mathcal V}$.
Roughly speaking, we allow the replacement of $V$ by any neighborhood of $]X[_{V}$ in $V$ and we make the role of $P$ secondary (only existence is required).

Since we call our category a site, we must endow it with a topology which is actually defined by the pretopology of families of formal morphisms
$$
\xymatrix{X \ar@{^{(}->}[r] \ar@{=}[d] & P_{i} \ar[d]^{v_{i}} & P_{iK} \ar[l] \ar[d]^{v_{iK}} & V_{i} \ar[l] \ar@{_{(}->}[d] \\ X \ar@{^{(}->}[r] & P & P_{K} \ar[l] & V \ar[l]
}
$$
where $V_{i}$ is open in $V$ and $]X[_{V} \subset \cup V_{i}$ (this is a \emph{standard} site).

Since the formal scheme plays a very loose role in the theory, we usually denote by $(X, V)$ an overconvergent variety and write $(f,u)$ for a morphism.

We use the general formalism of \emph{restricted} category (also called \emph{localized} or \emph{comma} or \emph{slice} category) to define relative overconvergent sites.
First of all, we define an \emph{overconvergent presheaf} as a presheaf (of sets) $T$ on $\mathrm{An}^\dagger_{/\mathcal V}$.
If we are given an overconvergent presheaf $T$, we may consider the restricted site $\mathrm{An}^\dagger_{/T}$.
An object is a section $s$ of $T$ on some overconvergent variety $(X, V)$ but we like to see $s$ as a morphism from (the presheaf represented by) $(X,V)$ to $T$.
 We will then say that $(X,V)$ is a \emph{(overconvergent) variety over $T$}.
A morphism between varieties over $T$ is just a morphism of overconvergent varieties which is compatible with the given sections.
The above pretopology is still a pretopology on $\mathrm{An}^\dagger_{/T}$ and we denote by $T_{\mathrm{An}^\dagger}$ the corresponding topos.
As explained by David Zureick-Brown in his thesis (see \cite{ZureickBrown10} and \cite{ZureickBrown14*}), one may as well replace $\mathrm{An}^\dagger_{/T}$ by any fibered category over $\mathrm{An}^\dagger_{/\mathcal V}$.
This is necessary if one wishes to work with algebraic stacks instead of algebraic varieties.

As a first example, we can apply our construction to the case of a representable sheaf $T := (X, V)$.
Another very important case is the following: we are given an overconvergent variety $(C, O)$ and an algebraic variety $X$ over $C$.
Then, we define the overconvergent sheaf $X/O$ as follows: a section of $X/O$ is a variety $(X', V')$ over $(C, O)$ with a given factorization $X' \to X \to C$ (this definition extends immediately to algebraic spaces - or even algebraic stacks if one is ready to work with fibered categories).
Alternatively, if we are actually given a variety $(X, V)$ over $(C, O)$, we may also consider the overconvergent presheaf $X_{V}/O$: a section is a variety $(X', V')$ over $(C, O)$ with a given factorization $X' \to X \to C$ which extends to \emph{some} factorization $(X', V') \to (X, V) \to (C, O)$.
Note that we only require the \emph{existence} of the second factorization.
In other words, $X_{V}/O$ is the image presheaf of the natural map $(X, V) \to X/O$.
An important theorem (more precisely its corollary 2.5.12 in \cite{LeStum11}) states that, if we work only with \emph{good} overconvergent varieties, then there exists an isomorphism of topos $(X_{V}/O)_{\mathrm{An}^\dagger} \simeq (X/O)_{\mathrm{An}^\dagger}$ when we start from a \emph{geometric} situation

\begin{align} \label{geom}
\xymatrix{X \ar@{^{(}->}[r] \ar[d]^f & P \ar[d]^v & P_{K} \ar[l] \ar[d]^{v_{K}} & V \ar[l] \ar[d]^u \\ C \ar@{^{(}->}[r] & S & S_{K} \ar[l] & O \ar[l]
}
\end{align}

with $P$ proper and smooth around $X$ over $S$ and $V$ a neighborhood of the tube of $X$ in $P_{K} \times_{S_{K}} O$ (and $(C,O)$ is good).

If we are given a morphism of overconvergent presheaves $v : T' \to T$,  we will also say that $T'$ is a \emph{(overconvergent) presheaf over} $T$.
It will induce a morphism of topos $v_{\mathrm{An}^\dagger} : T'_{\mathrm{An}^\dagger} \to T_{\mathrm{An}^\dagger}$.
We will often drop the index $\mathrm{An}^\dagger$ and keep writing $v$ instead of $v_{\mathrm{An}^\dagger}$.
Also, we will usually write the inverse image of a sheaf $\mathcal F$ as $\mathcal F_{|T'}$ when there is no ambiguity about $v$.
Note that there will exist a triple of adjoint functors $v_{!}, v^{-1}, v_{*}$ with $v_{!}$ exact.

For example, any morphism $(f, u) : (Y, W) \to (X, V)$ of overconvergent varieties will give rise to a morphism of topos
$$
(f,u)_{\mathrm{An}^\dagger} : (Y, W)_{\mathrm{An}^\dagger} \to (X, V)_{\mathrm{An}^\dagger}.
$$
It will also induce a morphism of overconvergent presheaves $f_{u} : Y_{W}/O \to X_{V}/O$ giving rise to a morphism of topos $f_{u\mathrm{An}^\dagger} : (Y_{W}/O)_{\mathrm{An}^\dagger} \to (X_{V}/O)_{\mathrm{An}^\dagger}$.
Finally, if $(C,O)$ is an overconvergent variety, then any morphism $f : Y \to X$ of algebraic varieties over $C$ induces a morphism of overconvergent presheaves $f : Y/O \to X/O$ giving rise to a morphism of topos $f_{\mathrm{An}^\dagger} : (Y/O)_{\mathrm{An}^\dagger} \to (X/O)_{\mathrm{An}^\dagger}$.

If we are given an overconvergent variety $(X, V)$, there exists a \emph{realization} map (morphism of topos)
$$
\xymatrix@R0cm{(X, V)_{\mathrm{An}^\dagger} \ar[r]^{\varphi} & ]X[_{V \mathrm{an}} \\ (X, V') & \ar@{|->}[l] ]X[_{V'}}
$$
where $]X[_{V \mathrm{an}}$ denotes the category of sheaves (of sets) on the analytic variety $]X[_{V}$ (which has a section $\psi$).
Now, if $T$ is any overconvergent presheaf and $(X, V)$ is a variety over $T$, then there exists a canonical morphism $(X, V) \to T$.
Therefore, if $\mathcal F$ is a sheaf on $T$, we may consider its restriction $\mathcal F_{|(X,V)}$ which is a sheaf on $(X, V)$.
We define the \emph{realization} of $\mathcal F$ on $(X, V$) as
$$
\mathcal F_{X,V} := \varphi_{V*}\left(\mathcal F_{|(X,V)}\right)
$$
(we shall simply write $\mathcal F_{V}$ in practice unless we want to emphasize the role of $X$).
As one might expect, the sheaf $\mathcal F$ is completely determined by its realizations $\mathcal F_{V}$ and the transition morphisms $]f[_{u}^{-1}\mathcal F_{V} \to \mathcal F_{V'}$ obtained by functoriality whenever $(f, u) : (X', V') \to (X, V)$ is a morphism over $T$.

We will need below the following result :

\begin{prop}\label{cart}
If we are given a \emph{cartesian} diagram of overconvergent presheaves (with a representable upper map)
$$
\xymatrix{
(X', V') \ar[rr]^{(f,u)} \ar[d]^{s'}  && (X, V) \ar[d]^s 
\\
T' \ar[rr]^v && T,
}
$$
and $\mathcal F'$ is a sheaf on $T'$, then
$$
(v_{*}\mathcal F')_{V} = ]f[_{u*}\mathcal F'_{V'}.
$$
\end{prop}

\begin{proof}
Since the diagram is cartesian, we have (this is formal)
$$
s^{-1}v_{*}\mathcal F' = (f,u)_{*}s'^{-1} \mathcal F'.
$$
It follows that 
\begin{align*}
(v_{*}\mathcal F')_{V} &= \varphi_{V*}s^{-1}v_{*}\mathcal F' 
\\
&= \varphi_{V*}(f,u)_{*}s'^{-1}\mathcal F' 
\\
&= ]f[_{u*}\varphi_{V'*}s'^{-1}\mathcal F'
\\
&=  ]f[_{u*}\mathcal F'_{V'}. \qedhere
\end{align*}
\end{proof}

If $(X, V)$ is an overconvergent variety, we will denote by $i_{X} : ]X[_{V} \hookrightarrow V$ the inclusion map.
Then, if $T$ is an overconvergent presheaf, we define the structural sheaf of $\mathrm{An}^\dagger_{/T}$ as the sheaf $\mathcal O^\dagger_{T}$ whose realization on any $(X, V)$ is $i_{X}^{-1} \mathcal O_{V}$.
An $\mathcal O^\dagger_{T}$-module $E$ will also be called a \emph{(overconvergent) module} on $T$.
As it was the case for sheaves of sets, the module $E$ is completely determined by its realizations $E_{V}$ and the transition morphisms
\begin{align} \label{transm}
]f[_{u}^{\dagger}E_{V} := i_{X'}^{-1}u^*i_{X*} E_{V }\to E_{V'}
\end{align}
obtained by functoriality whenever $(f, u) : (X', V') \to (X, V)$ is a morphism over $T$.
A module on $T$ is called an \emph{(overconvergent) isocrystal} if all the transition maps \eqref{transm} are actually isomorphisms  (used to be called a crystal in \cite{LeStum11}).
We will denote by
$$
\mathrm{Isoc}^\dagger(T) \subset \mathcal O^\dagger_{T}\mathrm{-Mod}
$$
the full subcategory made of all isocrystals on $T$ (used to be denoted by $\mathrm{Cris}^\dagger(T)$ in \cite{LeStum11}).
Be careful that inclusion is only right exact in general.

If we are given a morphism of overconvergent presheaves $v : T' \to T$ then the functors $v_{!}, v^{-1}, v_{*}$ preserve modules (we use the same notation $v_{!}$ for sheaves of sets and abelian groups:
 this should not create any confusion) and $v^{-1}$ preserves isocrystals.

One can show that a module on $T$ is locally finitely presented if and only if it is an isocrystal with coherent realizations.
We will denote their category by $\mathrm{Isoc}^\dagger_{\mathrm{coh}}(T)$ (be careful that it only means that the realizations are coherent: $\mathcal O^\dagger_{T}$ is not a coherent ring in general).
In the case $T = X/S_{K}$ and $\mathrm{Char}(K) = 0$, this is equivalent to Berthelot's original definition 2.3.6 in \cite{Berthelot96c*} of an overconvergent isocrystal.

Back to our examples, it is not difficult to see that, when $(X,V)$ is an overconvergent variety, the realization functor induces an equivalence of categories
$$
\mathrm{Isoc}^\dagger(X,V) \simeq i_{X}^{-1}\mathcal O_{V}\mathrm{-Mod}
$$
between isocrystals on $(X, V)$ and $i_{X}^{-1}\mathcal O_{V}$-modules.
Now, if $(X, V)$ is a variety over an overconvergent variety $(C,O)$ and
$$
p_{1}, p_{2} : (X, V \times_{O} V) \to (X, V)
$$
denote the projections, we define an \emph{overconvergent stratification} on an $i_{X}^{-1}\mathcal O_{V}$-module $\mathcal F$ as an isomorphism
$$
\epsilon : ]p_{2}[^\dagger \mathcal F \simeq ]p_{1}[^\dagger \mathcal F
$$
that satisfies the cocycle condition on triple products and the normalization condition along the diagonal.
They form an additive category $\mathrm{Strat}^\dagger(X,V/O)$ with cokernels and tensor products.
It is even an abelian category when $V$ is universally flat over $O$ in a neighborhood of $]X[_{V}$.
In any case, the realization functor will induce an equivalence 
$$
\mathrm{Isoc}^\dagger(X_{V}/O) \simeq \mathrm{Strat}^\dagger(X,V/O).
$$
We may also consider for $n \in \mathbb N$, the $n$-th infinitesimal neighborhood $V^{(n)}$ of $V$ in $V \times_{O} V$.
Then, a \emph{(usual) stratification} on an $i_{X}^{-1}\mathcal O_{V}$-module $\mathcal F$ is a compatible family of isomorphisms
$$
\epsilon^{(n)} : i_{X}^{-1}\mathcal O_{V^{(n)}} \otimes_{i_{X}^{-1}\mathcal O_{V}} \mathcal F \simeq \mathcal F \otimes_{i_{X}^{-1}\mathcal O_{V}} i_{X}^{-1}\mathcal O_{V^{(n)}}
$$
that satisfy the cocycle condition on triple products and the normalization condition along the diagonal.
Again, they form an additive category $\mathrm{Strat}(X,V/O)$ with cokernels and tensor products, and even an abelian category when $V$ is smooth over $O$ in a neighborhood of $]X[_{V}$.
There exists an obvious faithful functor
\begin{align} \label{map1}
\mathrm{Strat}^\dagger(X,V/O) \to \mathrm{Strat}(X,V/O).
\end{align}
Note that, \emph{a priori}, different overconvergent stratifications might give rise to the same usual stratification (and of course many usual stratifications will \emph{not} extend at all to an overconvergent one).
Finally, a \emph{connection} on an $i_{X}^{-1}\mathcal O_{V}$-module $\mathcal F$ is an $\mathcal O_{O}$-linear map
$$
\nabla : \mathcal F \to \mathcal F \otimes_{i_{X}^{-1}\mathcal O_{V}} i_{X}^{-1} \Omega^1_{V}
$$
that satisfies the Leibnitz rule.
Integrability is defined as usual.
They form an additive category $\mathrm{MIC}(X,V/O)$ and there exists again a faithful functor
\begin{align} \label{map2}
\mathrm{Strat}(X,V/O) \to \mathrm{MIC}(X,V/O)
\end{align}
 ($\nabla$ is induced by $\epsilon^{(1)} - \sigma$ where $\sigma$ switches the factors in $V \times_{O} V$).
When $V$ is smooth over $O$ in a neighbourhood of $]X[_{V}$ and $\mathrm{Char}(K) = 0$, then the functor \eqref{map2} is an equivalence.
Actually, both categories are then equivalent to the category of $i_{X}^{-1}\mathcal D_{V/O}$-modules.
In general, we will denote by $\mathrm{MIC}^\dagger(X,V/O)$ the image of the composition of the functors \eqref{map1} and \eqref{map2} and then call the connection \emph{overconvergent}
(and add an index $\mathrm{coh}$ when we consider only coherent modules).
Thus, there exists a realization functor
\begin{align} \label{map3}
\mathrm{Isoc}^\dagger(X_{V}/O) \to \mathrm{MIC}^\dagger(X,V/O)
\end{align}
which is faithful and essentially surjective (but not an equivalence in general).
In practice, we are interested in isocrystals on $X/O$ where $(C,O)$ is an overconvergent variety and $X$ is an algebraic variety over $C$.
We can localize in order to find a geometric realization $V$ for $X$ over $O$ such as \eqref{geom} and work directly on $(X,V)$: there exists an equivalence of categories
$$
\mathrm{Isoc}^\dagger(X/O) \simeq \mathrm{Isoc}^\dagger(X_{V}/O)
$$
that may be composed with \eqref{map3} in order to get the realization functor
$$
\mathrm{Isoc}^\dagger(X/O) \to \mathrm{MIC}^\dagger(X,V/O).
$$
In \cite{LeStum11}, we proved that, when $\mathrm{Char}(K) = 0$, it induces an equivalence
$$
\mathrm{Isoc}_{\mathrm{coh}}^\dagger(X/O) \simeq \mathrm{MIC}_{\mathrm{coh}}^\dagger(X,V/O)
$$
(showing in particular that the right hand side is independent of the choice of the geometric realization and that they glue).
We will extend this below to what we call \emph{constructible isocrystals}.

\section{Locally closed embeddings} \label{emb}

In this section, we fix an algebraic variety $X$ over $k$.
Recall that a (overconvergent) variety over $X/\mathcal M(K)$ (we will simply write $X/K$ in the future) is a pair made of an overconvergent variety $(X', V')$ and a morphism $X' \to X$.
In other words, it is a diagram
$$
\xymatrix{&& V' \ar[d] \\ X' \ar@{^{(}->}[r] \ar[d] & P' & P'_{K} \ar[l] \\ X && \quad }
$$
where $P'$ is a formal scheme.

We also fix a presheaf $T$ over $X/K$.
For example, $T$ could be (the presheaf represented by) an overconvergent variety $(X', V')$ over $X/K$.
Also, if $(C, O)$ is an overconvergent variety and $X$ is an algebraic variety over $C$, then we may consider the sheaf $T := X/O$ (see section \ref{bible}).
Finally, if we are given a morphism of overconvergent varieties $(X, V) \to (C, O)$, then we could set $T := X_{V}/O$ (see section \ref{bible} again).

Finally, we also fix an open immersion $\alpha : U \hookrightarrow X$ and denote by $\beta : Z \hookrightarrow X$ the embedding of a closed complement.
Actually, in the beginning, we consider more generally a locally closed embedding $\gamma : Y \hookrightarrow X$.

\begin{dfn}
The \emph{restriction} of $T$ to $Y$ is the inverse image
$$
T_{Y} := (Y/K) \times_{(X/K)} T
$$
of $T$ over $Y/K$.
We will still denote by $\gamma : T_{Y} \hookrightarrow T$ the corresponding map.
When $\mathcal F$ is a sheaf on $T$, the \emph{restriction} of $T$ to $Y$ is $\mathcal F_{|Y} := \gamma^{-1}\mathcal F$.
\end{dfn}

For example, if $T = (X', V')$ is a variety over $X/K$, then $T_{Y} = (Y', V')$ where $Y'$ is the inverse image of $Y$ in $X'$.
Also, if $(C, O)$ is an overconvergent variety, $X$ is an algebraic variety over $C$ and $T = X/O$, then $T_{Y} = Y/O$.
Finally, if we are given a morphism of overconvergent varieties $(X, V) \to (C, O)$ and $T = X_{V}/O$, then we will have $T_{Y} = Y_{V}/O$.

If $(X, V)$ is an overconvergent variety, we may consider the morphism of overconvergent varieties $(\gamma, \mathrm{Id}_{V}) : (Y, V) \hookrightarrow (X, V)$.
We will then denote by $]\gamma[_{V} : ]Y[_{V} \hookrightarrow ]X[_{V}$ or simply $]\gamma[$ if there is no ambiguity, the corresponding map on the tubes.
Recall that $]\gamma[$ is the inclusion of an analytic domain.
This is an open immersion when $\gamma$ is a closed embedding and conversely (we use Berkovich topology).

The next result generalizes proposition 3.1.10 of \cite{LeStum11}.

\begin{prop}\label{direct}
Let $(X', V')$ be an overconvergent variety over $T$ and $\gamma' : Y' \hookrightarrow X'$ the inclusion of the inverse image of $Y$ inside $X'$.
If $\mathcal F$ is a sheaf on $T_{Y}$, then
$$
(\gamma_{*}\mathcal F)_{X',V'} = ]\gamma'[_{*}\mathcal F_{Y',V'}.
$$
\end{prop}

\begin{proof}
Using corollary 2.4.15 of \cite{LeStum11}, this follows from proposition \ref{cart}.
\end{proof}

Since we will use it in some of our examples, we should also mention that $\mathrm R^i\gamma_{*} E = 0$ for $i > 0$ when $E$ is an isocrystal with coherent realizations.

We can work out very simple examples right now.
We will do our computations on the overconvergent variety 
$$
\mathbb P^1_{k/K} := \mathbb P^{1}_{k} \hookrightarrow \widehat {\mathbb P}^{1}_{\mathcal V} \leftarrow \mathbb P^{1,\mathrm{an}}_{K}.
$$
We consider first the open immersion $\alpha : \mathbb A^{1}_{k} \hookrightarrow \mathbb P^{1}_{k}$ and the structural sheaf $\mathcal O^\dagger_{\mathbb A^{1}_{k}/K}$.
If we let $i : \mathbb D(0, 1^+) \hookrightarrow \mathbb P^{1,\mathrm{an}}_{K}$ denote the inclusion map, we have
$$
\mathrm R\Gamma(\mathbb P^{1}_{k/K}, \alpha_{*} \mathcal O^\dagger_{\mathbb A^{1}_{k}/K}) = \mathrm R\Gamma(\mathbb P^{1,\mathrm{an}}_{K}, i_{*}i^{-1}\mathcal O_{\mathbb P^{1,\mathrm{an}}_{K}}) = K[t]^\dagger :=  \cup_{\lambda > 1} K\{t/\lambda\}
$$
(functions with radius of convergence (strictly) bigger than one at the origin).
 
On the other hand, if we start from the inclusion $\beta : \infty \hookrightarrow \mathbb P^{1}_{k}$ and let $j : \mathbb D(\infty, 1^-) \hookrightarrow \mathbb P^{1,\mathrm{an}}_{K}$ denote the inclusion map, we have
$$
\mathrm R\Gamma(\mathbb P^{1}_{k/K}, \beta_{*} \mathcal O^\dagger_{\infty/K}) = \mathrm R\Gamma(\mathbb P^{1,\mathrm{an}}_{K}, j_{*}j^{-1}\mathcal O_{\mathbb P^{1,\mathrm{an}}_{K}}) = K[1/t]^\mathrm{an} := \cap_{\lambda > 1} K\{\lambda/t\}
$$
(functions with radius of convergence at least one at infinity).

\begin{cor} \label{resex}
We have
\begin{enumerate}
\item $
\gamma_{\mathrm{An}^\dagger}^{-1} \circ \gamma_{\mathrm{An}^\dagger *} = \mathrm{Id},
$
\item
if $\gamma' : Y' \hookrightarrow X$ is another locally closed embedding with $Y \cap Y' = \emptyset$, then
$$
\gamma_{\mathrm{An}^\dagger}^{-1} \circ \gamma'_{\mathrm{An}^\dagger *} = 0.
$$
\end{enumerate}
\end{cor}

Alternatively, one may say that if $\mathcal F$ is a sheaf on $T_{Y}$, we have
$$
(\gamma_{*}\mathcal F)_{|Y} = \mathcal F \quad \mathrm{and} \quad (\gamma_{*}\mathcal F)_{|Y'} = 0.
$$

The first assertion of the corollary means that $\gamma_{\mathrm{An}^\dagger}$ is an embedding of topos (direct image is fully faithful).
Actually, from the fact that $Y$ is a subobject of $X$ in the category of varieties, one easily deduces that $T_{Y}$ is a subobject of $T$ in the overconvergent topos and $\gamma_{\mathrm{An}^\dagger}$ is therefore an \emph{open} immersion of topos.
Note also that the second assertion applies in particular to open and closed complements (both ways): in particular, these functors \emph{cannot} be used to glue along open and closed complements.
We will need some refinement.

We focus now on the case of an \emph{open} immersion $\alpha : U \hookrightarrow X$ which gives rise to a closed embedding on the tubes.

\begin{prop}
The functor $\alpha_{\mathrm{An}^\dagger*} : T_{U\mathrm{An}^\dagger} \to T_{\mathrm{An}^\dagger}$ is exact and preserves isocrystals.
\end{prop}

\begin{proof}
This is not trivial but can be proved exactly as in Corollary 3.1.12 and Proposition 3.3.15 of \cite{LeStum11} (which is the case $T = X/O$).
\end{proof}
 
The following definition is related to rigid cohomology with compact support (recall that $\beta : Z \hookrightarrow X$ denotes the embedding of a closed complement of $U$):

\begin{dfn}
If $\mathcal F$ is a sheaf of abelian groups on $T$, then
$$
\underline \Gamma_{U}\mathcal F = \ker(\mathcal F \to \beta_{*}\mathcal F_{|Z})
$$
is the subsheaf of \emph{sections of $\mathcal F$ with support in $U$}.
\end{dfn}

If we denote by $\mathcal U$ the \emph{closed} subtopos of $T_{\mathrm{An}^\dagger}$ which is the complement of the open topos $T_{Z\mathrm{An}^\dagger}$, then $\underline \Gamma_{U}$ is the same thing as the functor $\mathcal H^0_{\mathcal U}$ of sections with support in $\mathcal U$.
With this in mind, the first two assertions of the next proposition below are completely formal.
One may also show that the functor $\mathcal F \mapsto \mathcal F/\beta_{!}\beta^{-1}\mathcal F$ is an exact left adjoint to $\underline \Gamma_{U}$ ; it follows that $\underline \Gamma_{U}$ preserves injectives.

Actually, we should use the open/closed formalism only in the classical situation.
Recall (see \cite{Iversen86}, section II.6 for example for these kinds of things) that if $i : W \hookrightarrow V$ is a closed embedding of topological spaces, then $i_{*}$ has a right adjoint $i^!$ (and one usually sets $\underline \Gamma_{W}:= i_{*}i^!$) which commutes with direct images.
If $(X, V)$ is an overconvergent variety, we know that $]\alpha[ : ]U[ \hookrightarrow ]X[$ is a closed embedding and we may therefore consider the functors $]\alpha[^{!}$ and $\underline \Gamma_{]U[}$.

\begin{prop} \label{gammu}
\begin{enumerate}
\item The functor $\underline \Gamma_{U}$ is left exact and preserves modules,
\item if $\mathcal F$ is a sheaf of abelian groups on $T$, then there exists a distinguished triangle
$$
\mathrm R \underline \Gamma_{U} \mathcal F \to \mathcal F \to \mathrm R \beta_{*}\mathcal F_{|Z} \to ,
$$
\item \label{locgam} if $(X', V')$ is a variety over $T$ and $\alpha' : U' \hookrightarrow X'$ denotes the immersion of the inverse image of $U$ into $X'$, we have
$$
(\mathrm R \underline \Gamma_{U}E)_{V'} = \mathrm R \underline \Gamma_{]U'[_{V'}}E_{V'}
$$
for any \emph{isocrystal} $E$ on $T$.
\end{enumerate}
\end{prop}

\begin{proof}
The first assertion follows immediately from the fact that all the functors involved ($\beta^{-1}$, $\beta_{*}$ and $\ker$) do have these properties.
The second assertion results from the fact that the map $\mathcal F \to \beta_{*}\mathcal F_{|Z}$ is surjective when $\mathcal F$ is an injective sheaf (this is formal).
In order to prove the last assertion, it is sufficient to remember (this is a standard fact) that there exists a distinguished triangle
$$
\mathrm R \underline \Gamma_{]U'[_{V'}}E_{V'} \to E_{V'} \to \mathrm R ]\beta'[_{*}]\beta[^{-1}E_{V'} \to 
$$
where $\beta' : Z' \hookrightarrow X'$ denotes the inverse image of the inclusion of a closed complement of $U$.
 Since $E$ is an isocrystal, we have $(E_{|Z})_{Z',V'} = ]\beta'[^{-1}E_{X',V'}$.
\end{proof}

Note that the second assertion means that there exists an exact sequence
$$
0 \to \underline \Gamma_{U} \mathcal F \to \mathcal F \to \beta_{*}\mathcal F_{|Z} \to \mathrm R^1\underline \Gamma_{U} \mathcal F \to 0
$$
and that $\mathrm R^{i} \beta_{*}\mathcal F_{|Z} = \mathrm R^{i+1}\underline \Gamma_{U} \mathcal F$ for $i > 0$.

We can do the exercise with $\alpha : \mathbb A^{1}_{k} \hookrightarrow \mathbb P^{1}_{k}$ and $\beta : \infty \hookrightarrow \mathbb P^{1}_{k}$ as above.
We obtain
$$
\mathrm R\Gamma(\mathbb P^{1}_{k/K}, \mathrm R \underline \Gamma_{\mathbb A^1_{k}}\mathcal O^\dagger_{\mathbb P^{1}_{k}/K})
= [K \to K[1/t]^{\mathrm {an}}] 
 = (K[1/t]^{\mathrm {an}}/K)[-1].
$$

Since realization does not commute with the inverse image in general, we need to consider the following functor:

\begin{lem}
If $\mathcal F$ is a sheaf on $T$, then the assignment
$$
(X', V') \mapsto (j_{U}^\dagger \mathcal F)_{V'} := ]\alpha'[_{*}]\alpha'[^{-1}\mathcal F_{V'},
$$
where $\alpha' : U' \hookrightarrow X'$ denotes the immersion of the inverse image of $U$ into $X'$, defines a sheaf on $T$.
\end{lem}

\begin{proof}
We give ourselves a morphism $(f, u) : (X'', V'') \to (X', V')$ over $T$, we denote by $g : U'' \to U'$ the map induced by $f$ on the inverse images of $U$ into $X'$ and $X''$ respectively, and by $\alpha'' : U'' \hookrightarrow X''$ the inclusion map.
And we consider the cartesian diagram (forgetful functor to algebraic varieties is left exact)
$$
\xymatrix{(U'', V'') \ar@{^{(}->}[r] \ar[d]^{(g,u)}& (X'', V'') \ar[d]^{(f,u)} \\ (U', V') \ar@{^{(}->}[r] & (X', V')}
$$
which gives rise to a  cartesian diagram (tube is left exact)
$$
\xymatrix{]U''[_{V''} \ar@{^{(}->}[r] \ar[d]^{]g[_{u}}& ]X''[_{V''} \ar[d]^{]f[_{u}}\\ ]U'[_{V'} \ar@{^{(}->}[r] & ]X'[_{V'} }.
$$
Since $]\alpha'[$ is a closed embedding, we have $]f[_{u}^{-1} \circ ]\alpha'[_{*} = ]\alpha''[_{*} \circ ]g[_{u}^{-1}$ and there exists a canonical map
$$
]f[_{u}^{-1}]\alpha'[_{*}]\alpha'[^{-1}\mathcal F_{V'} = ]\alpha''[_{*}]g[_{u}^{-1}]\alpha'[^{-1}\mathcal F_{V'} = ]\alpha''[_{*}]\alpha''[^{-1}]f[_{u}^{-1}\mathcal F_{V''} \to ]\alpha''[_{*}]\alpha''[^{-1}\mathcal F_{V''}.\qedhere
$$
\end{proof}

\begin{dfn}
If $\mathcal F$ is a sheaf on $T$, then $j^\dagger_{U}\mathcal F$ is the sheaf of \emph{overconvergent sections of $\mathcal F$ around $U$}.
\end{dfn}

\begin{prop}
\begin{enumerate}
\item
The functor $j_{U}^\dagger$ is exact and preserves isocrystals,
\item
if $E$ is an isocrystal on $T$, we have $j_{U}^\dagger E = \alpha_{*}\alpha^{-1} E$.
\end{enumerate}
\end{prop}

\begin{proof}
Exactness can be checked on realizations.
But, if $(X',V')$ is a variety over $T$ and $\alpha' : U' \hookrightarrow X'$ denotes the immersion of the inverse image of $U$ in $X'$, then we know the exactness of $]\alpha'[_{*}$ (because $]\alpha'[$ is a closed embedding) and $]\alpha'[^{-1}$.
The second part of the first assertion is a consequence of the second assertion which follows from the fact that $(\alpha^{-1}E)_{V'} = ]\alpha'[^{-1} E_{V'}$ when $E$ is an isocrystal.
\end{proof}

Note that the canonical map $j_{U}^\dagger \mathcal F \to \alpha_{*}\alpha^{-1} \mathcal F$ is still bijective when $\mathcal F$ is a sheaf of \emph{Zariski type} (see Definition 4.6.1\footnote{the comment following definition 4.6.1 in \cite{LeStum11} is not correct and Lemma 4.6.2 is only valid for an \emph{open} immersion} of \cite{LeStum11}) but there are important concrete situations where equality fails as we shall see right now.

In order to exhibit a counter example, we let again $\alpha : \mathbb A^{1}_{k} \hookrightarrow \mathbb P^{1}_{k}$ and $\beta : \infty \hookrightarrow \mathbb P^{1}_{k}$ denote the inclusion maps and consider the sheaf
$
\mathcal F := \beta_{*}\mathcal O^\dagger_{\infty/K}
$
which is \emph{not} an isocrystal (and not even of Zariski type).
Since $\alpha^{-1} \circ \beta_{*} = 0$, we have $\alpha_{*}\alpha^{-1}\mathcal F = 0$.
Now, let us denote by $i_{\xi} : \xi \hookrightarrow \mathbb P^{1,\mathrm{an}}_{K}$ the inclusion of the generic point of the unit disc (corresponding to the Gauss norm) and let $i : \mathbb D(0, 1^+) \hookrightarrow \mathbb P^{1,\mathrm{an}}_{K}$ and $j : \mathbb D(\infty, 1^-) \hookrightarrow \mathbb P^{1,\mathrm{an}}_{K}$ be the inclusion maps as above.
Let
$$
\mathcal R := \left\{\sum_{n \in \mathbb Z} a_{n}t^n, \quad \left\{\begin{array}{c} \exists \lambda > 1, \lambda^na_{n} \to 0\ \mathrm{for}\ n \to + \infty \\ \forall \lambda > 1, \lambda^na_{n} \to 0 \ \mathrm{for}\ n \to -\infty \end{array} \right. \right\}
$$
be the \emph{Robba ring} (functions that converge on some open annulus of outer radius one at infinity).
Then, one easily sees that
$$
(j^\dagger_{\mathbb A^1_{k}}\beta_{*}\mathcal O_{\infty/K}^\dagger)_{\mathbb P^{1}_{k/K}} = i_{*}i^{-1}j_{*}\mathcal O_{\mathbb D(0, 1^-)} = i_{\xi*}\mathcal R
$$
so that $j^\dagger_{\mathbb A^1_{k}}\mathcal F \neq 0$.
This computation also shows that
$$
\mathrm R\Gamma(\mathbb P^{1}_{k/K}, j^\dagger_{\mathbb A^1_{k}}\beta_{*}\mathcal O^\dagger_{\infty/K}) = \mathcal R.
$$

We now turn to the study of the \emph{closed} embedding $\beta : Z \hookrightarrow X$ which requires some care (as we just experienced, the direct image of an isocrystal needs not be an isocrystal).

The following definition has to do with cohomology with support in a closed subset.

\begin{dfn}
For any sheaf of abelian groups $\mathcal F$ on $T$,
$$
\underline \Gamma^\dagger_{Z} \mathcal F := \ker\left(\mathcal F \to \alpha_{*}\mathcal F_{|U}\right)
$$
is the subsheaf of \emph{overconvergent sections of $\mathcal F$ with support in $Z$}.
\end{dfn}

We will do some examples below when we have more material at our disposal.

As above, if we denote by $\mathcal Z$ the closed subtopos of $T_{\mathrm{An}^\dagger}$ which is the complement of the open topos $T_{U\mathrm{An}^\dagger}$, then $\underline \Gamma^\dagger_{Z}$ is the same thing as the functor $\mathcal H^0_{\mathcal Z}$ of sections with support in $\mathcal Z$.
This is the approach taken by David Zureick-Brown in \cite{ZureickBrown10} and \cite{ZureickBrown14*} in order to define cohomology with support in $Z$ on the overconvergent site.
The next proposition is completely formal if one uses Zureick-Brown's approach.
Also, as above, one may prove that $\underline \Gamma^\dagger_{Z}$ preserves injectives because the functor $\mathcal F \mapsto \mathcal F/\alpha_{!}\alpha^{-1}\mathcal F$ is an exact left adjoint.

\begin{prop}\label{gamdag}
\begin{enumerate}
\item
The functor $\underline \Gamma^\dagger_{Z}$ is left exact and preserves modules.
\item
If $\mathcal F$ is an abelian sheaf on $T$, then there exists a  distinguished triangle
$$
0 \to \mathrm R\underline \Gamma^\dagger_{Z} \mathcal F \to \mathcal F \to \alpha_{*}\mathcal F_{|U} \to.
$$
\end{enumerate}
\end{prop}

We will also show below that $\underline \Gamma^\dagger_{Z}$ preserves isocrystals.

\begin{proof}
As in the proof of proposition \ref{gammu}, the first assertion follows from the fact that all the functors involved (and the kernel as well) are left exact and preserve overconvergent modules.
And similarly the second one is a formal consequence of the definition because $\alpha_{*}$ and $\alpha^{-1}$ both preserve injectives (they both have an exact left adjoint) and the map $\mathcal F \to \alpha_{*}\mathcal F_{|U}$ is an epimorphism when $\mathcal F$ is injective (standard).
\end{proof}

Note that the last assertion of the proposition means that there exists an exact sequence
$$
0 \to \underline \Gamma^\dagger_{Z} \mathcal F \to \mathcal F \to \alpha_{*}\mathcal F_{|U} \to \mathrm R^1\underline \Gamma^\dagger _{Z} \mathcal F \to 0.
$$
and that $\mathrm R^{i}\underline \Gamma^\dagger_{Z} \mathcal F = 0$ for $i > 1$.

Before going any further, we want to stress out the fact that $\beta^{-1}$ has an adjoint $\beta_{!}$ on the left in the category of all modules (or abelian groups or even sets with a light modification) but $\beta_{!}$ does not preserve isocrystals in general.
Actually, we always have $(\beta_{!}\mathcal F)_{X',V'} = 0$ unless the morphism $X' \to X$ factors through $Z$ (recall that we use the coarse topology on the algebraic side).
Again, the workaround consists in working directly with the realizations.
If $j : W \hookrightarrow V$ is an open immersion of topological spaces, then $j^{-1}$ has an adjoint $j_{!}$ on the left also (on sheaves of abelian groups or sheaves of sets with a light modification).
This is an exact functor that commutes with inverse images (see \cite{Iversen86}, II.6 again).
Now, if $(X, V)$ is an overconvergent variety, then $]\beta[ : ]Z[ \hookrightarrow ]X[$ is an \emph{open} immersion and we may consider the functor $]\beta[_{!}$.

\begin{lem} \label{lmbdg}
If $\mathcal F$ is a sheaf (of sets or abelian groups) on $T_{Z}$, then the assignment
$$
(X', V') \mapsto (\beta_{\dagger} \mathcal F)_{X',V'} := ]\beta'[_{!} \mathcal F_{Z',V'},
$$
where $\beta' : Z' \hookrightarrow X'$ denotes the embedding of the inverse image of $Z$ into $X'$, defines a sheaf on $T$.
Moreover, if $E$ is an isocrystal on $T_{Z}$, then $\beta_{\dagger}E$ is an isocrystal on $T$.
\end{lem}

\begin{proof}
As above, we consider a morphism $(f, u) : (X'', V'') \to (X', V')$ over $T$.
We denote by $h : Z'' \to Z'$ the map induced by $f$ on the inverse images of $Z$ into $X'$ and $X''$ respectively, and by $\beta'' : Z'' \hookrightarrow X''$ the inclusion map.
We have a cartesian diagram
$$
\xymatrix{(Z'', V'') \ar@{^{(}->}[r] \ar[d]^{(h,u)}& (X'', V'') \ar[d]^{(f,u)} \\ (Z', V') \ar@{^{(}->}[r] & (X', V')}
$$
giving rise to a cartesian diagram
$$
\xymatrix{]Z''[_{V''} \ar@{^{(}->}[r] \ar[d]^{]h[_{u}}& ]X''[_{V''} \ar[d]^{]f[_{u}}\\ ]Z'[_{V'} \ar@{^{(}->}[r] & ]X'[_{V'} }.
$$
It follows that there exists a canonical map
$$
]f[_{u}^{-1} ]\beta'[_{!} \mathcal F_{V'} = ]\beta''[_{!}  ]h[_{u}^{-1} \mathcal F_{V'} \to ]\beta''[_{!} \mathcal F_{V''}
$$
as asserted.
We consider now an isocrystal $E$ and we want to show that 
$$
]f[_{u}^{\dagger} ]\beta'[_{!} E_{V'} \simeq ]\beta''[_{!} E_{V''}.
$$
This immediately follows from the equality (which is formal)
$$
i_{X''}^{-1}\mathcal O_{V''} \otimes_{i_{X''}^{-1}u^{-1}\mathcal O_{V'}} ]\beta''[_{!}  ]h[_{u}^{-1} E_{V'} = ]\beta''[_{!} \left(i_{Z''}^{-1}\mathcal O_{V''} \otimes_{i_{Z''}^{-1}u^{-1}\mathcal O_{V'}}  ]h[_{u}^{-1} E_{V'}\right). \qedhere
$$
\end{proof}

\begin{dfn} $\beta_{\dagger} \mathcal F$ is the \emph{overconvergent direct image} of $\mathcal F$.
\end{dfn}

Note that there exists two flavors of $\beta_{\dagger}$: for sheaves of sets and for sheaves of abelian groups.
Whichever we consider should be clear from the context.

\begin{prop} \label{invim}
\begin{enumerate}
\item
If  $\mathcal F$ a sheaf on $T_{Z}$, then
\begin{enumerate}
\item
$(\beta_{\dagger} \mathcal F)_{|Z} = \mathcal F$,
\item $(\beta_{\dagger} \mathcal F)_{|U} = 0$,
\item if $E$ is an isocrystal on $T$, then
\begin{align} \label{omdag}
\mathcal Hom(\beta_{\dagger} \mathcal F, E) = \beta_{*}\mathcal Hom(\mathcal F, \beta^{-1}E),
\end{align}
\item
there exists a short exact sequence
\begin{align} \label{fonsec}
0 \to \beta_{\dagger} \mathcal F \to \beta_{*} \mathcal F \to j_{U}^\dagger \beta_{*} \mathcal F \to  0,
\end{align}
\end{enumerate}
\item
$\beta_{\dagger}$ is a fully faithful exact functor that preserve isocrystals, and the induced functor
$$
\beta_{\dagger} : \mathrm{Isoc}^\dagger(T_{Z}) \to \mathrm{Isoc}^\dagger(T)
$$
is left adjoint to 
$$
\beta^{-1} : \mathrm{Isoc}^\dagger(T) \to \mathrm{Isoc}^\dagger(T_{Z}).
$$
\end{enumerate}
\end{prop}

\begin{proof}
As usual, if $(X',V')$ is a variety over $T$, then we denote by $\alpha' : U' \hookrightarrow X'$ and $\beta' : Z' \hookrightarrow X'$ the inclusions of the inverse images of $U$ and $Z$ respectively.
 
When $(X', V')$ is an overconvergent variety over $T_{Z}$, then we will have $]\beta'[ = \mathrm{Id}$, and when $(X', V')$ is an overconvergent variety over $T_{U}$ then $]\beta'[ = \emptyset$.
We obtain the first two assertions.
When $E$ is an isocrystal on $T$, we have an isomorphism (this is standard)
$$
\mathcal Hom(]\beta'[_{!} \mathcal F_{V'}, E_{V'}) = ]\beta'[_{*}\mathcal Hom(\mathcal F_{V'}, ]\beta'[^{-1}E_{V'}),
$$
from which the third assertion follows.
Also, there exists  a short exact sequence
$$
0 \to ]\beta'[_{!} \mathcal F_{Z',V'} \to ]\beta'[_{*} \mathcal F_{Z',V'} \to ]\alpha'[_{*}]\alpha'[^{-1}]\beta'[_{*} \mathcal F_{Z',V'} \to  0
$$
which provides the fourth assertion.

Full faithfulness and exactness of $\beta_{\dagger}$ follow from the full faithfulness and exactness of $]\beta'[_{!}$ for all $(X', V')$.
The fact that $\beta_{\dagger}$ preserves isocrystals was proved in lemma \ref{lmbdg}.
The last assertion may be obtained by taking global sections on the equality \eqref{omdag}.
\end{proof}

We can also mention that there exists a distinguished triangle
$$
\beta_{\dagger} \mathcal F \to \mathrm R \beta_{*} \mathcal F \to j_{U}^\dagger \mathrm R \beta_{*} \mathcal F \to.
$$

Now, we prove that the exact sequence \eqref{fonsec} is universal:

\begin{prop} \label{exthom}
If $\mathcal F'$ and $\mathcal F''$ are modules on $T_{Z}$ and $T_{U}$ respectively, then any extension
$$
0 \to \beta_{\dagger} \mathcal F' \to \mathcal F \to \alpha_{*} \mathcal F'' \to  0
$$
is a pull back of the fundamental extension \eqref{fonsec} through a unique morphism $\alpha_{*} \mathcal F'' \to j_{U}^\dagger \beta_{*} \mathcal F'$.
\end{prop}

\begin{proof}
We know that $\beta^{-1}\alpha_{*}\mathcal F'' = 0$ and it follows that
$$
\mathrm{Hom}(\alpha_{*}\mathcal F'', \beta_{*}\mathcal F') = \mathrm{Hom}(\beta^{-1}\alpha_{*}\mathcal F'', \mathcal F') = 0.
$$
This being true for any sheaves, we see that actually, we have $\mathrm R\mathrm{Hom}(\alpha_{*}\mathcal F'', \mathrm R\beta_{*}\mathcal F') = 0$.
It formally follows that that $\mathrm R^{i}\mathrm{Hom}(\alpha_{*}\mathcal F'', \beta_{*}\mathcal F') = 0$ for $i \leq 1$.
As a consequence, we obtain a canonical isomorphism
\begin{align} \label{fdis}
\mathrm{Hom}(\alpha_{*}\mathcal F'', j_{U}^\dagger\beta_{*}\mathcal F') \simeq \mathrm{Ext}(\alpha_{*}\mathcal F'', \beta_{\dagger}\mathcal F').
\end{align}
This is exactly the content of our assertion.
\end{proof}

We should observe that we always have $\mathrm{Hom}(\alpha_{*}\mathcal F'', \beta_{\dagger}\mathcal F') = 0$.
 However, it is \emph{not} true that $\mathrm{Ext}(\alpha_{*}\mathcal F'', \beta_{\dagger}\mathcal F') = 0$ in general.
This can happen because $\beta_{\dagger}$ does not preserve injectives (although it is exact).

The overconvergent direct image is related to overconvergent support as follows:

\begin{prop} \label{isex}
If $E$ is an isocrystal on $T$, then
$$
\underline \Gamma^\dagger_{Z} E = \beta_{\dagger} E_{|Z}
$$
and, for all $i > 0$, $\mathrm R^i \underline \Gamma^\dagger_{Z} E =0$.
\end{prop}

\begin{proof}
Recall from proposition \ref{gamdag} that there exists an exact sequence
$$
0 \to \underline \Gamma^\dagger_{Z} E \to E \to \alpha_{*}E_{|U} \to \mathrm R^1\underline \Gamma^\dagger _{Z} E \to 0
$$
and that $\mathrm R^i \underline \Gamma^\dagger_{Z} E =0$ for $i > 1$.
Now, let $(X', V')$ be a variety over $T$. Denote by $\beta' : Z' \hookrightarrow X', \alpha' : U' \hookrightarrow X'$ the embeddings of the inverse images of $Z$ and $U$ into $X'$.
There exists a short exact sequence (standard again)
$$
0 \to ]\beta'[_{!}]\beta'[^{-1}E_{V'} \to E_{V'} \to ]\alpha'[_{*}]\alpha'[^{-1}E_{V'} \to 0.
$$
Since $E$ is an isocrystal, we have $(\alpha_{*}E_{|U})_{V'} = ]\alpha'[_{*}]\alpha'[^{-1}E_{V'}$.
It follows that $(\mathrm R^1 \underline \Gamma^\dagger_{Z} E)_{V'} =0$ and we also see that
$$
(\underline \Gamma^\dagger_{Z} E)_{V'} = ]\beta'[_{!}]\beta'[^{-1} E_{V'} = (\beta_{\dagger} E_{|Z})_{V'}. \qedhere
$$
\end{proof}

Note that the proposition is still valid for sheaves of Zariski type and not merely for isocrystals.
Be careful however that $\beta_{\dagger} E \neq \underline \Gamma^\dagger_{Z} \beta_{*}E$ in general even when $E$ is an isocrystal on $T_{Z}$.
With our favorite example in mind, we have $\underline \Gamma^\dagger_{Z} \beta_{*}\mathcal O^\dagger_{\infty/K} = \beta_{*}\mathcal O^\dagger_{\infty/K} \neq \beta_{\dagger}\mathcal O^\dagger_{\infty/K}$ as our computations below will show.

\begin{cor} \label{exact}
The functor $\underline \Gamma^\dagger_{Z}$ preserves isocrystals and the induced functor
$$
\underline \Gamma^\dagger_{Z} : \mathrm{Isoc}^\dagger(T) \to \mathrm{Isoc}^\dagger(T)
$$
is exact.
Moreover, if $E$ is an isocrystal on $T$, then there exists a short exact sequence
$$
0 \to \underline \Gamma^\dagger_{Z} E \to E \to j^\dagger_{U}E \to 0.
$$
\end{cor}
 
We might as well write this last short exact sequence as
 $$
0 \to \beta_{\dagger}E_{|Z} \to E \to \alpha_{*}E_{|U} \to 0.
$$

As promised above, we can do an example and consider the closed embedding $\beta : \infty \hookrightarrow \mathbb P^{1}_{k}$ again.
We compute $\beta_{\dagger}\mathcal O^\dagger_{\infty/K} = \underline \Gamma^\dagger_{\infty} \mathcal O^\dagger_{\mathbb P^1_{k}/K}$.
We have
$$
\mathrm R\Gamma(\mathbb P^{1}_{k/K}, \beta_{\dagger}\mathcal O^\dagger_{\infty/K}) = [K \to K[t]^\dagger] = \left(K[t]^\dagger/K\right)[-1].
$$
We can also remark that the (long) exact sequence obtained by applying $\mathrm R\Gamma(\mathbb P^{1}_{k/K}, -)$ to the fundamental short exact sequence
$$
0 \to \beta_{\dagger}\mathcal O^\dagger_{\infty/K} \to \beta_{*} \mathcal O^\dagger_{\infty/K} \to j_{U}^\dagger \beta_{*} \mathcal O^\dagger_{\infty/K} \to  0
$$
reads
\begin{align} \label{Robseq}
0 \to K[1/t]^{\mathrm{an}} \to \mathcal R \to K[t]^\dagger/K \to 0.
\end{align}

\begin{cor} \label{eqemb}
\begin{enumerate}
\item
The functors $\alpha_{*}$ and $\alpha^{-1}$ induce an equivalence between isocrystals on $T_{U}$ and isocrystals on $T$ such that $\underline \Gamma^\dagger_{Z}E = 0$ (or $j^\dagger_{U}E= E$).
\item
The functors $\beta_{\dagger}$ and $\beta^{-1}$ induce an equivalence between isocrystals on $T_{Z}$ and isocrystals on $T$ such that $\underline \Gamma^\dagger_{Z}E = E$ (or $j^\dagger_{U}E=0$).
\end{enumerate}
\end{cor}

\begin{proof}
If $E''$ is an isocrystal on $T_{U}$, then $\alpha_{*}E''$ is an isocrystal on $T$ and therefore $\underline \Gamma^\dagger_{Z}\alpha_{*}E'' = \beta_{\dagger}\beta^{-1} \alpha_{*} E'' = 0$. And conversely, if $E$ is an isocrystal on $T$ such that $\underline \Gamma^\dagger_{Z}E = 0$, then $E = j^\dagger_{U}E = \alpha_{*}\alpha^{-1}E$.
This shows the first part.

Now, if $E'$ is an isocrystal on $T_{Z}$, then $\beta_{\dagger}E'$ is an isocrystal on $T$ and therefore $\underline \Gamma^\dagger_{Z}\beta_{\dagger}E' = \beta_{\dagger}\beta^{-1}\beta_{\dagger} E' = \beta_{\dagger} E'$. And conversely, if $E$ is an isocrystal on $T$ such that $\underline \Gamma^\dagger_{Z}E = E$, then $E = \beta_{\dagger}\beta^{-1}E$.
\end{proof}

We can also make the functor of sections with support in an open subset come back into the picture:

\begin{cor}
If $E$ is an isocrystal on $T$, then there exists a distinguished triangle
$$
\mathrm R \underline \Gamma_{U} E \to j^\dagger_{U} E \to  j^\dagger_{U} \mathrm R \beta_{*}E_{|Z} \to.
$$
\end{cor}

\begin{proof}
There exists actually a commutative diagram of distinguished triangles:
$$
\xymatrix{
& \underline \Gamma^\dagger_{Z} E \ar@{=}[r] \ar[d] & \underline \Gamma^\dagger_{Z} E \ar[d]
\\
\mathrm R \underline \Gamma_{U} E \ar[r] \ar@{=}[d] & E \ar[r] \ar[d] &  \mathrm R \beta_{*}E_{|Z} \ar[r] \ar[d] &
\\
\mathrm R \underline \Gamma_{U} E \ar[r] & j^\dagger_{U} E \ar[r] \ar[d] &  j^\dagger_{U} \mathrm R \beta_{*}E_{|Z} \ar[r] \ar[d] &.
\\
&&&
}
$$
More precisely, we know that the vertical triangles as well as the middle horizontal one are all distinguished.
The bottom one must be distinguished too.
\end{proof}

Back to our running example, we see that the long exact sequence obtained by applying $\mathrm R\Gamma(\mathbb P^{1}_{k/K}, -)$ to the distinguished triangle
$$
\mathrm R \underline \Gamma_{\mathbb A^1_{k}} \mathcal O^\dagger_{\mathbb P^1_{k}/K} \to j^\dagger_{\mathbb A^1_{k}} \mathcal O^\dagger_{\mathbb P^1_{k}/K} \to  j^\dagger_{\mathbb A^1_{k}} \mathrm R \beta_{*}\mathcal O^\dagger_{\infty/K} \to 
$$
reads
$$
0 \to K[t]^\dagger \to \mathcal R \to  K[1/t]^{\mathrm{an}}/K \to0.
$$

We can summarize the situation as follows:

\begin{enumerate}
\item There exists two triples of adjoint functors (up means left) :
$$
\xymatrix{
\mathcal O^\dagger_{T_{U}}\mathrm{-Mod}  \ar@<.5cm>@{^{(}->}[rr]^{\alpha_{!}} \ar@<-.5cm>@{^{(}->}[rr]^{\alpha_{*}} && \mathcal O^\dagger_{T}\mathrm{-Mod} \ar[ll]_-{\alpha^{-1}} \ar[rr]^-{\beta^{-1}} && \mathcal O^\dagger_{T_{Z}}\mathrm{-Mod}. \ar@<-.5cm>@{_{(}->}[ll]_{\beta{!}} \ar@<.5cm>@{_{(}->}[ll]_{\beta{*}}
}
$$
Moreover, $\alpha_{*}$ is \emph{exact} and \emph{preserves isocrystals} (and so do $\alpha^{-1}$ and $\beta^{-1}$).
\item There exists two functors with support (that preserve injectives)
$$
\xymatrix{
\underline \Gamma_{U} & \ar@(ul,dl)[] &  \mathcal O^\dagger_{T}\mathrm{-Mod} & \ar@(ur,dr)[] & \underline \Gamma^\dagger_{Z}.
}
$$
Moreover, $\Gamma^\dagger_{Z}$ \emph{preserves isocrystals} and is \emph{exact on isocrystals}.
\item There exists two other functors
$$
\xymatrix{
j^\dagger_{U} & \ar@(ul,dl)[] & \mathcal O^\dagger_{T}\mathrm{-Mod} &&  \mathcal O^\dagger_{T_{Z}}\mathrm{-Mod}. \ar@{_{(}->}[ll]_{\beta_{\dagger}}
}
$$
They are both \emph{exact} and \emph{preserve isocrystals} (but not injectives).
 If $E$ is an isocrystal on $T$, we have
 $$
 j^\dagger_{U}E = \alpha_{*} E_{|U} \quad \mathrm {and} \quad \underline \Gamma^\dagger_{Z} E = \beta_{\dagger}E_{|Z}.
 $$
\end{enumerate}

\section{Constructibility}

Recall that $K$ denotes a complete ultrametric field with ring of integers $\mathcal V$ and residue field $k$.
We let $X$ be an algebraic variety over $k$ and $T$ a (overconvergent) presheaf over $X/K$.
Roughly speaking, $T$ is some family of varieties $X'$ over $X$ which embed into a formal $\mathcal V$-scheme $P'$, together with a morphism of analytic $K$-varieties $V' \to P'_{K}$. 
A \emph{(overconvergent) module} $\mathcal F$ on $T$ is then a compatible family of $i_{X'}^{-1}\mathcal O_{V'}$-modules $\mathcal F_{V'}$, where $i_{X'} : ]X'[_{V'} \hookrightarrow V'$ denotes the inclusion of the tube (the reader is redirected to section \ref{bible} for the details).

\begin{dfn}
A module $\mathcal F$ on $T$ is said to be \emph{constructible} (with respect to $X$) if there exists a locally finite covering of $X$ by locally closed subvarieties $Y$ such that $\mathcal F_{|Y}$ is locally finitely presented.
\end{dfn}
 
Recall that a locally finitely presented module is the same thing as an isocrystal with coherent realizations.
It is important to notice however that a constructible module is \emph{not} necessarily an isocrystal (the transition maps might not be bijective).
We'll give an example later.

\begin{prop}
\begin{enumerate}
\item
Constructible modules on $T$ form an additive category which is stable under cokernel, extension, tensor product and internal Hom.
\item
Constructible isocrystals on $T$ form an additive category $\mathrm{Isoc}^\dagger_{\mathrm{cons}}(T)$ which is stable under cokernel, extensions and tensor product.
\end{enumerate}
\end{prop}

\begin{proof}
The analog to the first assertion for locally finitely presented modules is completely formal besides the internal Hom question that was proved in proposition 3.3.12 of \cite{LeStum11}.
The analog to the second assertion for all isocrystals was proved in corollary 3.3.9 of \cite{LeStum11}.
Since the restriction maps $\mathcal F \mapsto \mathcal F_{|Y}$ are exact and commute with tensor product and internal Hom, everything follows.
\end{proof}

Note however that $\mathcal Hom(E_{1}, E_{2})$ needs \emph{not} be an isocrystal (see example below) when $E_{1}$ and $E_{2}$ are two constructible isocrystals.
 
\begin{prop} \label{firstprop}
 Let $\mathcal F$ be a module on $T$.
\begin{enumerate}
\item \label{equiv} $\mathcal F$ is constructible if and only if there exists a locally finite covering by locally closed subvarieties $Y$ of $X$ such that $\mathcal F_{|Y}$ is constructible.
\item
If $T' \to T$ is any morphism of overconvergent presheaves and $\mathcal F$ is constructible, then $\mathcal F_{|T'}$ is constructible.
The converse also is true if $T' \to T$ is a covering.
\item \label{astre}Assume that $T$ is actually a presheaf on $X'/K$ for some $f : X' \to X$.
If $\mathcal F$ is constructible with respect to $X$, then it is also constructible with respect to $X'$.
\end{enumerate}
\end{prop}

\begin{proof}
The first assertion is an immediate consequence of the transitivity of locally finite coverings by locally closed subsets: if $X = \cup X_{i}$ and $X_{i} = \cup X_{ij}$ are such coverings, so is the covering $X = \cup X_{ij}$.

In order to prove the second assertion, note first that it is formally satisfied by locally finitely presented modules.
Moreover, if $Y$ is a locally closed subvariety of $X$, we have $(\mathcal F_{|T'})_{|Y} = (\mathcal F_{|Y})_{|T'_{Y}}$.
The result follows.

Finally, for the third assertion, if $X = \cup X_{i}$ is a locally finite covering by locally closed subvarieties, so is $X' = \cup f^{-1}(X_{i})$.
Moreover, by definition $\mathcal F_{|f^{-1}(X_{i})} = \mathcal F_{|X_{i}}$ and there is nothing to do.
\end{proof}

Together with corollary \ref{eqemb} above, the next proposition will allow us to move freely along a closed or open embedding when we consider constructible isocrystals (note that this is obviously wrong for overconvergent isocrystals with coherent realizations):

\begin{prop}
\begin{enumerate}
\item
If $\alpha : U \hookrightarrow X$ is an open immersion of algebraic varieties, then a module $\mathcal F''$ on $T_{U}$ is constructible if and only if $\alpha_{*}\mathcal F''$ is constructible.
\item
If $\beta : Z \hookrightarrow X$ is a closed embedding of algebraic varieties, then a module $\mathcal F'$ on $T_{Z}$ is constructible if and only if $\beta_{\dagger}\mathcal F'$ is constructible.
\end{enumerate}
\end{prop}

\begin{proof}
We may assume that $U$ and $Z$ are open and closed complements.
We saw in corollary \ref{resex} that $(\alpha_{*} \mathcal F'')_{|Z} = \mathcal F''$ and $(\alpha_{*} \mathcal F')_{|U} = 0$.
And we also saw in proposition \ref{invim} that $(\beta_{\dagger} \mathcal F')_{|Z} = \mathcal F'$ and $(\beta_{\dagger} \mathcal F')_{|U} = 0$.
\end{proof}

It is easy to see that the \emph{usual} dual to a constructible isocrystal is \emph{not} an isocrystal in general: if $\beta : Z \hookrightarrow X$ is a closed embedding of algebraic varieties and $E$ is an overconvergent isocrystal on $Z$ with coherent realizations, it follows from proposition \ref{invim} that
$$
(\beta_{\dagger}E)\check{} := \mathcal Hom(\beta_{\dagger}E, \mathcal O^\dagger_{T_{}}) = \beta_{*}\mathcal Hom(E, \mathcal O^\dagger_{T_{Z}}) = \beta_{*}E\check{},
$$
which is constructible but is not an isocrystal in general (as we saw in section \ref{emb}).
 
The next property is also very important because it allows the use of noetherian induction to reduce some assertions about constructible isocrystals to analogous assertions about overconvergent isocrystals with coherent realizations.

\begin{lem} \label{semloc}
A module $\mathcal F$ on $T$ is constructible if and only if there exists a closed subvariety $Z$ of $X$ such that, if $U := X \setminus Z$, then both $\mathcal F_{|Z}$ and $\mathcal F_{|U}$ are constructible.
We may even assume that $U$ is dense in $X$ and $\mathcal F_{|U}$ is locally finitely presented.
\end{lem}

\begin{proof}
The condition is sufficient thanks to assertion \ref{equiv} of Proposition \ref{firstprop}.
Conversely, if $\xi$ is a generic point of $X$, then there exists a locally closed subset $Y$ of $X$ such that $\xi \in Y$ and $\mathcal F_{|Y}$ is locally finitely presented.
And $Y$ contains necessarily an open neighborhood $U_{\xi}$ of $\xi$ in $X$.
We may choose $U := \cup U_{\xi}$.
\end{proof}

\begin{prop} \label{extension}
An isocrystal $E$ on $T$ is constructible if and only if there exists an exact sequence
\begin{align} \label{fdext}
0 \to \beta_{\dagger}E' \to E \to \alpha_{*} E'' \to 0 
\end{align}
where $E''$ (resp. $E'$) is a constructible isocrystal on a closed subvariety $Z$ of $X$ (resp. on $U := X \setminus Z$) and $\beta : Z \hookrightarrow X$ (resp. $\alpha : U \hookrightarrow X$) denotes the inclusion map.
We may assume that $U$ is dense in $X$ and that $E''$ has coherent realizations.
\end{prop}

\begin{proof}
If we are given such and exact sequence, we may pull back along $\alpha$ and $\beta$ in order to obtain $E' \simeq E_{|Z}$ and $E'' \simeq E_{|U}$.
And conversely, we may set $E' :=  E_{|Z}$ and $E'' := E_{|U}$ in order to get such an exact sequence by proposition \ref{isex}.
\end{proof}

Note that this property is specific to constructible \emph{isocrystals} and that the analog for constructible modules is wrong.

It follows from proposition \ref{exthom} that any extension such as \eqref{fdext} comes from a unique morphism $\alpha_{*}E'' \to j^\dagger_{U}\beta_{*}E'$.
This is a classical glueing method and the correspondence is given by the following morphism of exact sequences
$$
\xymatrix{
0 \ar[r] & \beta_{\dagger}E' \ar[r] & \beta_{*}E' \ar[r] &j^\dagger_{U}\beta_{*}E' \ar[r] & 0
\\
0 \ar[r]  & \beta_{\dagger}E' \ar[r] \ar@{=}[u] & E \ar[u]  \ar[r] & \alpha_{*} E'' \ar[r] \ar[u] & 0.
}
$$
We can do the computations in the very special case of $\alpha : \mathbb A^{1}_{k} \hookrightarrow \mathbb P^{1}_{k}$ and $\beta : \infty \hookrightarrow \mathbb P^{1}_{k}$.
We have $E' = \mathcal O_{\infty/K}^\dagger \otimes_{K} H$ for some finite dimensional vector space $H$, and $E''$ is given by a finite free $K[t]^\dagger$-module $M$ of finite rank endowed with a (overconvergent) connection.
One can show that there exists a canonical isomorphism
\begin{align*}
\mathrm {Ext}(\alpha_{*}E'', \beta_{\dagger}E') & = \mathrm {Hom}(\alpha_{*}E'', j^\dagger_{U}\beta_{*}E')
\\
&= \mathrm{Hom}_{\nabla}(M, \mathcal R \otimes_{K} H)
\\
& = H^
{0}_{\mathrm{dR}}(\check M \otimes_{K[t]^\dagger} \mathcal R) \otimes_{K} H
\end{align*}
(the second identity is not trivial).
A slight generalization will give a classification of constructible isocrystals on smooth projective curves as in theorem 6.15 of \cite{LeStum14}.

\section{Integrable connections and constructibility}

In this section, we will give a more concrete description of constructible isocrystals in the case $T$ is representable by some overconvergent variety $(X,V)$, in the case $T = X_{V}/O$ where $(X,V)$ is a variety over some overconvergent variety $(C,O)$, and finally when $T = X/O$ where $X$ is a variety over $C$ (see section \ref{bible}).

\begin{dfn}
Let $(X, V)$ be an overconvergent variety.
An $i_{X}^{-1}\mathcal O_{V}$-module $\mathcal F$ is \emph{constructible} if there exists a locally finite covering of $X$ by locally closed subvarieties $Y$ such that $i_{Y}^{-1} i_{X*} \mathcal F$ is a coherent $i_{Y}^{-1}\mathcal O_{V}$-module.
\end{dfn}

Of course, we have $i_{Y}^{-1} i_{X*} \mathcal F = i_{Y\subset X}^{-1}\mathcal F$ if we denote by $i_{Y \subset X} : ]Y[_{V} \hookrightarrow ]X[_{V}$ the inclusion of the tubes.

\begin{prop} \label{consm}
Let $(X, V)$ be an overconvergent variety.
Then,
\begin{enumerate}
\item $\mathrm{Isoc}^\dagger_{\mathrm{cons}}(X,V)$ is an abelian subcategory of $\mathrm{Isoc}^\dagger(X,V)$,
\item the realization functor induces an equivalence between $\mathrm{Isoc}^\dagger_{\mathrm{cons}}(X,V)$ and the category of all constructible $i_{X}^{-1}\mathcal O_{V}$-modules.
\end{enumerate}
\end{prop}

\begin{proof}
It was shown in proposition 3.3.8 of \cite{LeStum11} that the realization functor induces an equivalence between $\mathrm{Isoc}^\dagger(X,V)$ and the category of all $i_{X}^{-1}\mathcal O_{V}$-modules.
And overconvergent isocrystals correspond to coherent modules.
The second assertion is an immediate consequence of these observations.
The first assertion then follows immediately from the analogous result about coherent modules.
\end{proof}

Recall that an $i_{X}^{-1}\mathcal O_{V}$-module may be endowed with an overconvergent stratification.
Then, we have:

\begin{prop} \label{relco}
Let $(X, V)$ be a variety over an overconvergent variety $(C,O)$.
\begin{enumerate}
\item  If $V$ universally flat over $O$ in a neighborhood of $]X[$, then $\mathrm{Isoc}^\dagger_{\mathrm{cons}}(X_{V}/O)$ is an abelian subcategory of $\mathrm{Isoc}^\dagger(X_{V}/O)$,
\item the realization functor induces an equivalence between $\mathrm{Isoc}^\dagger_{\mathrm{cons}}(X_{V}/O)$ and the category of constructible $i_{X}^{-1}\mathcal O_{V}$-modules $\mathcal F$ endowed with an overconvergent stratification.
\end{enumerate}
\end{prop}

\begin{proof}
According to proposition 3.5.3 of \cite{LeStum11}, its corollary and proposition 3.5.5 of \cite{LeStum11}, the proof goes exactly as in proposition \ref{consm}.
\end{proof}

The next corollary is valid if we work with \emph{good} overconvergent varieties (which we may have assumed from the beginning).

\begin{cor}
If $(C, O)$ is an (good) overconvergent variety and $X$ is an algebraic variety over $C$, then $\mathrm{Isoc}^\dagger_{\mathrm{cons}}(X/O)$ is an abelian subcategory of $\mathrm{Isoc}^\dagger(X/O)$.
\end{cor}

\begin{proof}
Using proposition 4.6.3 of \cite{LeStum11}, we may assume that $X$ has a geometric realization over $(C,O)$ and use the second part of proposition 3.5.8 in \cite{LeStum11}.
\end{proof}
 
We could have included a description of constructible isocrystal as modules endowed with an overconvergent stratification on some geometric realization of $X/O$ but we are heading towards a finer description (this is what the rest of this section is all about).

Recall that any overconvergent stratification will induce, by pull back at each level, a usual stratification.
This is a faithful construction and we want to show that it is actually \emph{fully} faithful when we work with constructible modules (in suitable geometric situations).
Thus, we have the following sequence of injective maps
$$
\mathrm{Hom}_{\mathrm{Strat}^\dagger}(\mathcal F, \mathcal G)  \hookrightarrow \mathrm{Hom}_{\mathrm{Strat}}(\mathcal F, \mathcal G)  \hookrightarrow \mathrm{Hom}(\mathcal F, \mathcal G)
$$
and we wonder wether the first one is actually bijective.
In order to do so, we will also have to study the injectivity of the maps in the sequence
$$
\mathrm{Ext}_{\mathrm{Strat}^\dagger}(\mathcal F, \mathcal G)  \to \mathrm{Ext}_{\mathrm{Strat}}(\mathcal F, \mathcal G)  \to \mathrm{Ext}(\mathcal F, \mathcal G).
$$

We start with the following observation:

\begin{prop} 
Let $(X, V)$ be a variety over an overconvergent variety $(C,O)$, $\alpha : U \hookrightarrow X$ the inclusion of an open subvariety of $X$ and $\beta : Z \hookrightarrow X$ the inclusion of a closed complement.
Let $\mathcal F'$ be an $i_{Z}^{-1}\mathcal O_{V}$-module and $\mathcal F''$ an $i_{U}^{-1}\mathcal O_{V}$-module.
Then a usual (resp. an overconvergent) stratification on the direct sum $]\beta[_{!} \mathcal F' \oplus  ]\alpha[_{*} \mathcal F''$ is uniquely determined by its restrictions to $\mathcal F'$ and $\mathcal F''$.
\end{prop}

\begin{proof}
Let us denote by
$$
\epsilon^{(n)}  = \left(\begin{array} {cc} ]\beta[_{!} \epsilon'^{(n)} & \varphi_{n} \\ \psi_{n} &  ]\alpha[_{*} \epsilon''^{(n)}\end{array} \right)
\quad \left(\mathrm{resp.}\ \epsilon  = \left(\begin{array} {cc} ]\beta[_{!} \epsilon' & \varphi \\ \psi &  ]\alpha[_{*} \epsilon''\end{array} \right)\right)
$$
the (resp. the overconvergent) stratification of $]\beta[_{!} \mathcal F' \oplus  ]\alpha[_{*} \mathcal F''$ (recall that the maps $]\beta[_{!}$ and $]\alpha[_{*}$ are fully faithful).
Then, the maps
$$
\varphi_{n} : i_{X}^{-1}\mathcal O_{V^{(n)}} \otimes_{i_{X}^{-1}\mathcal O_{V}} ]\alpha[_{*} \mathcal F'' \to ]\beta[_{!} \mathcal F' \otimes_{i_{X}^{-1}\mathcal O_{V}} i_{X}^{-1}\mathcal O_{V^{(n)}}
$$
and
$$
\psi_{n} : i_{X}^{-1}\mathcal O_{V^{(n)}} \otimes_{i_{X}^{-1}\mathcal O_{V}}  ]\beta[_{!} \mathcal F' \to ]\alpha[_{*} \mathcal F''  \otimes_{i_{X}^{-1}\mathcal O_{V}} i_{X}^{-1}\mathcal O_{V^{(n)}}
$$
$$
(\mathrm{resp.}\ \varphi:  ]p_{2}[^\dagger  ]\alpha[_{*} \mathcal F'' \simeq ]p_{1}[^\dagger ]\beta[_{!} \mathcal F' \quad \mathrm{and} \quad \psi :  ]p_{2}[^\dagger ]\beta[_{!} \mathcal F' \simeq ]p_{1}[^\dagger  ]\alpha[_{*} \mathcal F'')
$$
are necessarily zero as one may see by considering the fibres (resp. and using the fact that $p_{i}^\dagger$ commutes with $]\alpha[_{*}$ and $]\beta[_{!}$).
\end{proof}

We keep the assumptions and the notations of the proposition for a while and assume that $\mathcal F'$ and $\mathcal F''$ are both endowed with a usual (resp. an overconvergent) stratification.
From the general fact that
$$
\mathrm{Hom}( ]\beta[_{!} \mathcal F', ]\alpha[_{*} \mathcal F'') = 0\quad \mathrm{and} \quad \mathrm{Hom}(]\alpha[_{*} \mathcal F'', ]\beta[_{!} \mathcal F') = 0,
$$
we can deduce that
$$
\mathrm{Hom}_{\mathrm{Strat}}( ]\beta[_{!} \mathcal F', ]\alpha[_{*} \mathcal F'') = 0\quad \mathrm{and} \quad \mathrm{Hom}_{\mathrm{Strat}}(]\alpha[_{*} \mathcal F'', ]\beta[_{!} \mathcal F') = 0
$$
$$
(\mathrm{resp.}\ \mathrm{Hom}_{\mathrm{Strat}^\dagger}( ]\beta[_{!} \mathcal F', ]\alpha[_{*} \mathcal F'') = 0\quad \mathrm{and} \quad \mathrm{Hom}_{\mathrm{Strat}^\dagger}(]\alpha[_{*} \mathcal F'', ]\beta[_{!} \mathcal F') = 0).
$$
Since we also know that
$$
\mathrm{Ext}( ]\beta[_{!} \mathcal F', ]\alpha[_{*} \mathcal F'') = 0,
$$
we can deduce the following result from the proposition:

\begin{cor} \label{spliov1}
If $\mathcal F'$ and $\mathcal F''$ are both endowed with a usual (resp. an overconvergent) stratification, then we have
$$
\mathrm{Ext}_{\mathrm{Strat}}( ]\beta[_{!} \mathcal F', ]\alpha[_{*} \mathcal F'') = 0 \quad (\mathrm{resp.}\ \mathrm{Ext}_{\mathrm{Strat}^\dagger}( ]\beta[_{!} \mathcal F', ]\alpha[_{*} \mathcal F'') = 0).
$$
\end{cor}

Alternatively, it means that any short exact sequence of $i_{Z}^{-1}\mathcal O_{V}$-modules (resp. with a usual stratification, resp. with an overconvergent stratification)
$$
\xymatrix{
0 \ar[r]  & ]\alpha[_{*} \mathcal F'' \ar[r] & \mathcal F  \ar[r] & ]\beta[_{!} \mathcal F' \ar[r] & 0
}
$$
splits (and the splitting is compatible with the extra structure).

From the proposition, we may also deduce the following:

\begin{cor} \label{spliov2}
If $\mathcal F'$ and $\mathcal F''$ are both endowed with a usual (resp. an overconvergent) stratification, then the following map is (resp. maps are) injective
$$
\mathrm{Ext}_{\mathrm{Strat}}( ]\alpha[_{*} \mathcal F'', ]\beta[_{!} \mathcal F')  \hookrightarrow \mathrm{Ext}(  ]\alpha[_{*} \mathcal F'', ]\beta[_{!} \mathcal F')
$$
$$
\left(\mathrm{resp.}\ \mathrm{Ext}_{\mathrm{Strat}^\dagger}( ]\alpha[_{*} \mathcal F'', ]\beta[_{!} \mathcal F')  \hookrightarrow \mathrm{Ext}_{\mathrm{Strat}}( ]\alpha[_{*} \mathcal F'', ]\beta[_{!} \mathcal F')  \hookrightarrow \mathrm{Ext}(  ]\alpha[_{*} \mathcal F'', ]\beta[_{!} \mathcal F')\right).
$$
\end{cor}

Alternatively, it means that if $\mathcal F$ is an $i_{X}^{-1}\mathcal O_{V}$-module with a usual (resp. an overconvergent) stratification, \emph{and} if the exact sequence of $i_{X}^{-1}\mathcal O_{V}$-modules
$$
\xymatrix{
0 \ar[r]  & ]\beta[_{!} \mathcal F_{|]Z[}\ar[r] & \mathcal F  \ar[r] & ]\alpha[_{*} \mathcal F_{|]U[} \ar[r] & 0
}
$$
splits, then the splitting is always compatible with the (resp. the overconvergent) stratifications.

We are now ready to prove our main result:

\begin{prop} \label{fulfait}
Let
$$
\xymatrix{X \ar@{^{(}->}[r] \ar[d]^f & P \ar[d]^v & P_{K} \ar[l] \ar[d]^{v_{K}} & V \ar[l] \ar[d]^u \\ C \ar@{^{(}->}[r] & S & S_{K} \ar[l] & O \ar[l]
}
$$
be a formal morphism of overconvergent varieties with $f$ quasi-compact, $v$ smooth at $X$, $O$ locally separated and $V$ a good neighborhood of $X$ in $P_{K} \times_{S_{K}} O$.
If $\mathcal F$ and $\mathcal G$ are two constructible $i_{X}^{-1}\mathcal O_{V}$-modules endowed with an overconvergent stratification, then
$$
\mathrm{Hom}_{\mathrm{Strat}^\dagger}(\mathcal F, \mathcal G)  \simeq \mathrm{Hom}_{\mathrm{Strat}}(\mathcal F, \mathcal G).
$$
\end{prop}

\begin{proof}
Since we know that the map is injective, we may rephrase the assertion as follows:
we are given a morphism $\varphi : \mathcal F \to \mathcal G$ of constructible $i_{X}^{-1}\mathcal O_{V}$-modules and we have to show that $\varphi$ is actually compatible with the overconvergent stratifications.
This question is clearly local on $O$ which is locally compact.
We may therefore assume that the image of $O$ in $S_{K}$ is contained in some $S'_{K}$ with $S'$ quasi-compact.
We may then pull back the diagram along $S' \to S$ and assume that $X$ is finite dimensional (use assertion \ref{astre} of proposition \ref{firstprop}).
This will allow us to use noetherian induction.

We know (use for example propositions \ref{semloc} and \ref{relco}) that there exists a dense open subset $U$ of $X$ such that the restrictions $\mathcal F''$ and $\mathcal G''$ to $U$ of $\mathcal F$ and $\mathcal G$ are coherent.
Moreover, it was shown in corollary 3.4.10 of \cite{LeStum11} that the proposition is valid for $\mathcal F''$ and $\mathcal G''$ on $U$.
Let us denote as usual by $\alpha : U \hookrightarrow X$ the inclusion map.
Since $]\alpha[_{*}$ is fully faithful, we see that the proposition is valid for $]\alpha[_{*}\mathcal F''$ and $]\alpha[_{*}\mathcal G''$.
In other words, we have a bijection
\begin{align} \label{pf1}
\mathrm{Hom}_{\mathrm{Strat}^\dagger}(]\alpha[_{*}\mathcal F'', ]\alpha[_{*}\mathcal G'')  \simeq \mathrm{Hom}_{\mathrm{Strat}}(]\alpha[_{*}\mathcal F'', ]\alpha[_{*}\mathcal G'').
\end{align}
We denote now by $\beta : Z \hookrightarrow X$ the inclusion of a closed complement of $U$ and let $\mathcal F'$ and $\mathcal G'$ be the restrictions of $\mathcal F$ and $\mathcal G$ to $Z$.
And we observe the following commutative diagram:
$$
\xymatrix{
0  \ar[d] & 0  \ar[d] \\
\mathrm{Hom}_{\mathrm{Strat}^\dagger}(]\alpha[_{*}\mathcal F'', \mathcal G) \ar[d] \ar@{^{(}->}[r] & \mathrm{Hom}_{\mathrm{Strat}}(]\alpha[_{*}\mathcal F'', \mathcal G) \ar[d] \\ 
\mathrm{Hom}_{\mathrm{Strat}^\dagger}(]\alpha[_{*}\mathcal F'', ]\alpha[_{*}\mathcal G'')  \ar[d] \ar[r]^-\simeq & \mathrm{Hom}_{\mathrm{Strat}}(]\alpha[_{*}\mathcal F'', ]\alpha[_{*}\mathcal G'')  \ar[d] \\ 
\mathrm{Ext}_{\mathrm{Strat}^\dagger}(]\alpha[_{*}\mathcal F'', ]\beta[_{!}\mathcal G') \ar@{^{(}->}[r] & \mathrm{Ext}_{\mathrm{Strat}}(]\alpha[_{*}\mathcal F'', ]\beta[_{!}\mathcal G').
}
$$
The columns are exact because $\mathrm{Hom}(]\alpha[_{*}\mathcal F'', ]\beta[_{!}\mathcal G') = 0$, the bottom map is injective thanks to corollary \ref{spliov2} and the middle map is the isomorphism \eqref{pf1}.
It follows from the five lemma (or an easy diagram chasing) that the upper map is necessarily bijective: we have
\begin{align} \label{pf2}
\mathrm{Hom}_{\mathrm{Strat}^\dagger}(]\alpha[_{*}\mathcal F'', \mathcal G)  \simeq \mathrm{Hom}_{\mathrm{Strat}}(]\alpha[_{*}\mathcal F'', \mathcal G'').
\end{align}

We turn now to the other side: by induction, the proposition is valid for $\mathcal F'$ and $\mathcal G'$ on $Z$, and since $]\beta[_{!}$ is fully faithful, it also holds for $]\beta[_{!}\mathcal F'$ and $]\beta[_{!}\mathcal G'$.
Hence, we have
\begin{align} \label{pf3}
\mathrm{Hom}_{\mathrm{Strat}^\dagger}(]\beta[_{!}\mathcal F', ]\beta[_{!}\mathcal G')  \simeq \mathrm{Hom}_{\mathrm{Strat}}(]\beta[_{!}\mathcal F', ]\beta[_{!}\mathcal G').
\end{align}
Now, we consider the commutative square
$$
\xymatrix{
\mathrm{Hom}_{\mathrm{Strat}^\dagger}(]\beta[_{!}\mathcal F', ]\beta[_{!}\mathcal G') \ar[d]^-\simeq \ar[r]^-\simeq & \mathrm{Hom}_{\mathrm{Strat}}(]\beta[_{!}\mathcal F', ]\beta[_{!}\mathcal G') \ar[d]^-\simeq \\ 
\mathrm{Hom}_{\mathrm{Strat}^\dagger}(]\beta[_{!}\mathcal F',\mathcal G) \ar@{^{(}->}[r]  & \mathrm{Hom}_{\mathrm{Strat}}(]\beta[_{!}\mathcal F', \mathcal G).
}
$$
The vertical maps are bijective because $\mathrm{Hom}(]\beta[_{!}\mathcal F', ]\alpha[_{*}\mathcal G'') = 0$ and the upper map is simply the isomorphism \eqref{pf3}.
If follows that we have an isomorphism
\begin{align} \label{pf4}
\mathrm{Hom}_{\mathrm{Strat}^\dagger}(]\beta[_{!}\mathcal F', \mathcal G)  \simeq \mathrm{Hom}_{\mathrm{Strat}}(]\beta[_{!}\mathcal F', \mathcal G).
\end{align}
In order to end the proof, we will need to kill another obstruction.
Since the proposition holds for $]\alpha[_{*}\mathcal F''$ and \emph{any} constructible $\mathcal G$, then the following canonical map is necessarily injective:
\begin{align} \label{pf5}
\mathrm{Ext}_{\mathrm{Strat}^\dagger}(]\alpha[_{*}\mathcal F'', \mathcal G)  \hookrightarrow \mathrm{Ext}_{\mathrm{Strat}}(]\alpha[_{*}\mathcal F'', \mathcal G).
\end{align}
We consider now the commutative diagram with exact columns:
$$
\xymatrix{
0  \ar[d] & 0  \ar[d] \\
\mathrm{Hom}_{\mathrm{Strat}^\dagger}(]\alpha[_{*}\mathcal F'', \mathcal G) \ar[d] \ar[r]^\simeq & \mathrm{Hom}_{\mathrm{Strat}}(]\alpha[_{*}\mathcal F'', \mathcal G) \ar[d] \\
\mathrm{Hom}_{\mathrm{Strat}^\dagger}(\mathcal F, \mathcal G) \ar[d] \ar@{^{(}->}[r] & \mathrm{Hom}_{\mathrm{Strat}}(\mathcal F, \mathcal G) \ar[d] \\ 
\mathrm{Hom}_{\mathrm{Strat}^\dagger}(]\beta[_{\dagger}\mathcal F',\mathcal G)  \ar[d] \ar[r]^-\simeq & \mathrm{Hom}_{\mathrm{Strat}}(]\beta[_{\dagger}\mathcal F', \mathcal G)  \ar[d] \\ 
\mathrm{Ext}_{\mathrm{Strat}^\dagger}(]\alpha[_{*}\mathcal F'', \mathcal G) \ar@{^{(}->}[r] & \mathrm{Ext}_{\mathrm{Strat}}(]\alpha[_{*}\mathcal F'', \mathcal G).
}
$$
The horizontal isomorphisms are just \eqref{pf2} and \eqref{pf4} and the bottom injection is \eqref{pf5}.
It is then sufficient to apply the five lemma again.
\end{proof}

We may reformulate the statement of the proposition as follows:

\begin{cor}
The forgetful functor from constructible $i_{X}^{-1}\mathcal O_{V}$-modules endowed with an overconvergent stratification to $i_{X}^{-1}\mathcal O_{V}$-modules endowed with a usual stratification is fully faithful.
\end{cor}

It is also worth mentioning the following immediate consequence:

\begin{cor} \label{injbij}
If $\mathcal F$ and $\mathcal G$ are two constructible $i_{X}^{-1}\mathcal O_{V}$-modules endowed with an overconvergent stratification, then we have  an injective map
\begin{align} \label{sur}
\mathrm{Ext}_{\mathrm{Strat}^\dagger}(\mathcal F, \mathcal G)  \hookrightarrow \mathrm{Ext}_{\mathrm{Strat}}(\mathcal F, \mathcal G).
\end{align}
\end{cor}

It means that if
\begin{align} \label{ext}
\xymatrix{
0 \ar[r]  & \mathcal F \ar[r] & \mathcal G  \ar[r] & \mathcal H \ar[r] & 0
}
\end{align}
is a short exact sequence of constructible $i_{X}^{-1}\mathcal O_{V}$-modules endowed with an overconvergent stratification, then any splitting for the usual stratifications will be compatible with the overconvergent stratifications.
I strongly suspect that much more is actually true: if we are given an exact sequence \eqref{ext} of constructible $i_{X}^{-1}\mathcal O_{V}$-modules endowed with usual stratifications and if the stratifications of $\mathcal F'$ and $\mathcal F''$ are overconvergent, then the stratification of $\mathcal F$ should also be overconvergent.
In other words, the injective map \eqref{sur} would be an isomorphism.

If $(X, V)$ is a variety over an overconvergent variety $(C,O)$, then we will denote by
$$
\mathrm{MIC}^\dagger_{\mathrm{cons}}(X,V/O)
$$
the category of constructible $i_{X}^{-1}\mathcal O_{V}$-modules $\mathcal F$ endowed with an overconvergent connection (recall that it means that the connection extends to some overconvergent stratification).
Then, we can also state the following corollary:

\begin{cor} \label{equiV}
If $\mathrm{Char}(K) = 0$, then the realization functor induces an equivalence of categories
$$
\mathrm{Isoc}^\dagger_{\mathrm{cons}}(X_{V}/O) \simeq \mathrm{MIC}^\dagger_{\mathrm{cons}}(X,V/O).
$$
\end{cor}

As a consequence, we may observe that we will have, for a constructible isocrystal $E$ on $X_{V}/O$,
$$
\Gamma(X_{V}/O, E) \simeq H^{0}_{\mathrm{dR}}(E_{V}),
$$
and we expect the same to hold for higher cohomology spaces (we only know at this point that
$$
H^1(X_{V}/O, E) \subset H^{1}_{\mathrm{dR}}(E_{V})).
$$

Again, we need to work with good overconvergent varieties for the theorem to holds:

\begin{thm} \label{thm}
Assume that $\mathrm{Char}(K) = 0$ and that we are given a commutative diagram
\begin{align} \label{geomag}
\xymatrix{X \ar@{^{(}->}[r] \ar[d]^f & P \ar[d]^v & P_{K} \ar[l] \ar[d]^{v_{K}} & V \ar[l] \ar[d]^u \\ C \ar@{^{(}->}[r] & S & S_{K} \ar[l] & O \ar[l]
}
\end{align}
where $P$ is a formal scheme over $S$ which is proper and smooth around $X$ and $V$ is a neighborhood of the tube of $X$ in $P_{K} \times_{S_{K}} O$ (and $O$ is good in the neighborhood of $]C[$).
Then the realization functor induces an equivalence of categories
$$
\mathrm{Isoc}^\dagger_{\mathrm{cons}}(X/O) \simeq \mathrm{MIC}^\dagger_{\mathrm{cons}}(X,V/O)
$$
between constructible overconvergent isocrystals on $X/O$ and constructible $i_{X}^{-1}\mathcal O_{V}$-modules endowed with an overconvergent connection.
\end{thm}

\begin{proof}
Using the second assertion of proposition 3.5.8 in \cite{LeStum11}, this follows immediately from corollary \ref{equiV}.
\end{proof}

As a consequence of the theorem, we see that the notion of constructible module endowed with an overconvergent connection only depends on $X$ and \emph{not} on the choice of the geometric realization \eqref{geomag}.
It is likely that this could have been proven directly using Berthelot's technic of diagonal embedding.
However, we believe that our method is much more natural because functoriality is built-in.

%
%
\bibliographystyle{plain}
\addcontentsline{toc}{section}{Bibliography}
\bibliography{BiblioBLS}

\end{document}